\documentclass[11pt, reqno]{amsart}
\usepackage{amsmath,amssymb,amscd,mathrsfs,amscd}
\usepackage{graphics,verbatim, tikz}
\usetikzlibrary{snakes,matrix,arrows,patterns}

\usetikzlibrary{matrix,arrows,patterns}
\usepackage{marvosym}

\usepackage{amsmath,amstext,amsbsy,amssymb,amscd}
\usepackage{amsmath}
\usepackage{amsxtra}
\usepackage{amscd}
\usepackage{amsthm}
\usepackage{amsfonts}
\usepackage{graphicx}%
\usepackage{amssymb}
\usepackage{eucal}
\usepackage{color}
\usepackage{multirow}
\usepackage[enableskew,vcentermath]{youngtab}
\usepackage[all]{xy}
\newtheorem{theorem}{Theorem}[section]
\theoremstyle{plain}
\newtheorem{lemma}[theorem]{Lemma}
\newtheorem{proposition}[theorem]{Proposition}
\newtheorem{corollary}[theorem]{Corollary}

\newtheorem{definition}[theorem]{Definition}
\theoremstyle{definition}
\newtheorem{example}[theorem]{Example}

\theoremstyle{remark}
\newtheorem{remark}[theorem]{Remark}

\numberwithin{equation}{section}

\newcommand{\A}{\mathbb A}

\newcommand{\N}{\mathbb N}
\newcommand{\Q}{\mathbb Q}
\newcommand{\Z}{\mathbb Z}

\newcommand{\U}{\mathbf U}

\newcommand{\var}{\varepsilon}

{\vskip-\lastskip\medskip
  \noindent
  {\em #1.}\enspace
  }%
{\qed\par\medskip
  }

\newcommand{\Qv}{{\Q(v)}}
\newcommand{\Qvp}{{\Q(v)^\pi}}
\newcommand{\Qqp}{{\Q(q)^\pi}}
\newcommand{\Qp}{{\mathbb{Q}^\pi}}
\newcommand{\Zp}{{\Z^\pi}}
\newcommand{\Zpvvi}{{\Zp[v^{\pm 1}]}}

\newcommand{\Zpv}{{\Zp[v]}}

\newcommand{\cB}{{\mathcal{B}}}

\newcommand{\cL}{{\mathcal{L}}}

\newcommand{\barmap}{\overline{\phantom{x}}}

\newcommand{\height}{{\operatorname{ht}}}

\newcommand{\sgn}{{\operatorname{sgn}}}

\newcommand{\catO}{{\mathcal{O}}}

\newcommand{\zero}{{\bar{0}}}
\newcommand{\one}{{\bar{1}}}

\newcommand{\tK}{{\tilde{K}}}

\newcommand{\tJ}{{\tilde{J}}}

\newcommand{\te}{{\tilde{e}}}
\newcommand{\tf}{{\tilde{f}}}

\newcommand{\set}[1]{\left\{#1\right\}}
\newcommand{\parens}[1]{\left(#1\right)}
\newcommand{\ang}[1]{\left\langle#1\right\rangle}
\newcommand{\bra}[1]{\left[#1\right]}

\newcommand{\bbinom}[2]{\begin{bmatrix}#1 \\ #2\end{bmatrix}}

\renewcommand{\bar}[1]{\overline{#1}}

\newcommand{\bt}{{\mathbf t}}
\newcommand{\tT}{{\tilde T}}

\newcommand{\ff}{{\bf f}}

\newcommand{\fint}{{{}_\A\ff}}

\newcommand{\UU}{{\mathbf U}}
\newcommand{\hatU}{{\widehat \UU}}

\newcommand{\Up}{\UU^+}
\newcommand{\Um}{\UU^-}
\newcommand{\Uz}{\UU^0}
\newcommand{\Uint}{ { {}_\A\! \UU } }

\newcommand{\dotU}{{\dot{\mathbf U}}}
\newcommand{\dotUint}{ { {}_\A\! \dotU} }

\newcommand{\Tw}{{\mathfrak X}}

\begin{document}
\title{Quantum supergroups IV. The modified form}

\author[Clark]{Sean Clark}
\address{Department of Mathematics, University of Virginia, Charlottesville, VA 22904}
\email{
sic5ag@virginia.edu
}
\date{}

\begin{abstract}
We construct a canonical basis for a class of tensor product modules of a quantum covering group
associated to an Kac-Moody Lie superalgebra, and use these bases to construct a canonical basis for 
the modified form of a quantum covering group.
\end{abstract}

\maketitle
\setcounter{tocdepth}{1}

\tableofcontents
\section{Introduction}

\subsection{}
The modified form of the Drinfeld-Jimbo quantum group $\U$ is a non-unital algebra 
$\dotU$ which was first defined in type $A$ by 
Beilinson-Lusztig-MacPherson \cite{BLM}, and generalized to other types by Lusztig
\cite{L92}. 
The algebra $\dotU$ can be thought of as a version of $\U$ optimized for the study of 
weight $\UU$ modules, which are also $\dotU$-modules.
The modified form admits 
interesting higher-order structure that cannot be
constructed on the whole quantum group. 
In particular, the Lusztig-Kashiwara canonical basis of the half-quantum group 
$\UU^-$ \cite{L90,K91} can be extended to a canonical basis of $\dotU$ \cite{L92,K93},
and moreover $\dotU$ admits a categorification \cite{KL,Web13}.

Together with Hill and Wang \cite{CW,CHW1,CHW2}, we defined and systematically 
developed the theory of covering quantum groups
associated to an anisotropic Kac-Moody superalgebra. These algebras, whose
definition is inspired by the results in \cite{CW,HW}, combine the 
structure and representation theories of a classical quantum group and 
a quantum supergroup by using a ``half parameter'' $\pi$,
satisfying $\pi^2=1$, in the place of super signs;
one recovers the classical quantum group under the specialization $\pi=1$,
and the quantum supergroup under the specialization $\pi=-1$. 
Moreover, the half-quantum covering group and the 
integrable modules  admit (compatible) canonical bases 
which both generalize the Lusztig-Kashiwara canonical bases  and 
specialize to canonical bases for the corresponding
superalgebra upon setting $\pi=-1$.

The modified form of the rank one quantum covering group appeared in \cite{CW}, 
and there it is shown that this modified form retains many of the 
nice characteristics of its non-covering counterpart. In particular, it 
has a canonical basis which is structurally similar to the canonical basis 
of quantum $\mathfrak{sl}_2$, as well as a bilinear form, 
with respect to which the Chevalley generators are essentially adjoint.
A definition of the modified form in higher rank can be found in \cite{CFLW}.

Like their non-covering counterparts, quantum covering groups 
have strong ties to results in categorification.
Indeed, the quiver Hecke superalgebras
introduced by Kang, Kashiwara, and Tsuchioka \cite{KKT} (also see \cite{W,EKL})
provide an alternate categorification of the half-quantum group.
They also provide a categorification of certain types of 
half-quantum supergroups
as shown by Hill and Wang \cite{HW}.
These alternate categorifications glue together to a categorification
of the half-quantum covering group.
Moreover, Kang, Kashiwara, and Oh \cite{KKO} have shown 
that the cyclotomic quiver Hecke superalgebras categorify
the integrable modules (for rank one, also see \cite{EKL,HW}).
The modified form of the rank one quantum covering group 
was recently shown to admit a (super)categorification by Ellis and Lauda \cite{EL}. 

\subsection{}
The goal of this paper is to provide an intrinsic construction of a canonical
basis for the modified form of a quantum covering group, 
as well as develop auxiliary structure, such as a bilinear form, 
using parallel techniques to those
utilized in Lusztig's book \cite[Part IV]{L93}.

Using the structure and representation theory of quantum covering groups, 
we deduce a number of structural properties for the modified form
which parallel the structure theory of modified quantum groups.
Following Lusztig's construction of a bar-involution on tensor products, 
we use (a variant of) the quasi-$\mathcal R$-matrix defined in \cite{CHW1} 
to exhibit a bar-involution on a family of modules $\mathcal F$. 
This bar-involution can be used to 
construct a canonical basis for the modules which can be ``glued together''
to yield a canonical basis $\dot\cB$ for the modified form. This canonical
basis generalizes the canonical basis of Lusztig for the modified quantum groups, 
in the sense that $\dot\cB|_{\pi=1}$ is precisely Lusztig's canonical basis.
On the other hand, $\dot\cB|_{\pi=-1}$ furnishes a canonical basis for the modified
form of a Kac-Moody quantum supergroup of anisotropic type.

A surprising connection between
the modified forms of non-super and super quantum groups is that they can be realized
as rational forms of a complex algebra. This identification can be succintly realized
as an automorphism of the complexification of the modified form of the
quantum covering group. This map and the maps it induces on 
both the quantum covering group and the half-quantum covering group
are called twistors. 
These maps were defined and studied in \cite{CFLW} 
(extending work in rank one by Fan and Li \cite{FL}).

The twistor is remarkably useful for deducing properties of quantum covering groups
from the corresponding properties of quantum groups. 
However, the definition of the twistor
relies on a choice of total order on the simple roots, 
and its highly non-canonical nature
is not suitable for deducing the existence of
 a canonical basis of the quantum covering group. 
Nonetheless, it happens that the twistor matches 
the canonical basis of $\U^-$ with itself up to a factor of an integral 
power of $\sqrt{-1}$; see
\cite[Theorem 3.6]{CFLW}. We will build on the results in {\em loc. cit.}
to extend the notion of twistor to (complexifications of) the modules in $\mathcal F$.
These module twistors preserve the canonical bases (up to an integral power of $\sqrt{-1}$) as well.
As a consequence, we show that the canonical basis $\dot\cB$ of $\dotU$
is also preserved (up to scalar multiple) by the twistor.

We will also define a bilinear form on the modified quantum covering group
which generalizes the bilinear form on the half-quantum covering group. 
This bilinear form is shown to exhibit a number of
nice qualities; namely, it is invariant under some of the 
important automorphisms of the quantum covering group,
and it can be viewed as a limit of bilinear forms on the tensor product modules.
We note that in rank one, the bilinear form we shall define differs from the 
one defined in rank one in \cite{CW}, but they are intimately related; see Proposition \ref{prop:altdotbilform}. 
The bilinear form in {\em loc. cit.} is inadequate from the point of view of categorification (see
Remark \ref{remark:dagbilform}), but we will show this deficiency can be explained
by utilizing an additional symmetry of the quantum covering group to relate these forms.

Given the Ellis-Lauda categorification in rank one \cite{EL}
and the construction of odd knot invariants associated to finite type
quantum covering groups by Mikhaylov and Witten \cite{MW}, we expect that
quantum covering groups and their modified form will be an important setting
for obtaining further categorifications and categorical odd knot invariants.
We believe that the bilinear form in this paper will be
useful for constructing categorifications of modified quantum covering groups of higher rank.
However, we note that, in rank one, our bilinear form differs from the one proposed by Ellis
and Lauda, which appears to be a skew-version of the bilinear form defined in
Proposition \ref{prop:altdotbilform}. Nonetheless, the canonical bases
we construct agree with the earlier construction in \cite{CW},
which appear as indecomposable 1-morphisms in the Ellis-Lauda categorification.

We also note that there exist canonical bases of Lie superalgebras outside
of anisotropic Kac-Moody type. In \cite{CHW3}, canonical bases are constructed
for $\mathfrak{gl}(n|1)$ and $\mathfrak{osp}(2|2n)$ by a quantum shuffle approach. 
This technique tells us very little
about the relation of the canonical basis and the representation theory,
but it is expected that for polynomial modules,
the canonical basis will descend to a basis of the module (cf. \cite[\S 8]{CHW3}).
If this is true, then the canonical bases of {\em loc. cit.} should
have an extension to a suitable modified form (say, with
idempotents corresponding to polynomial weights). There is some evidence that
this should be possible; for example, see El Turkey and Kujawa's presentation 
of Schur superalgebras \cite{ETK}.

\subsection{}

The paper is structured in the following way. 
In Section 2, we will recall the
definitions and basic results about the quantum covering groups. We will
also prove some new technical results we will need in later sections.

In Section 3, we define the modified quantum group $\dotU$ and study its structure in detail.
In particular, we show that $\dotU$ is approximated by certain
tensor product modules in a way analogous to the approximation of $\U^-$ by simple
highest weight modules. 

In Section 4, we construct canonical bases for the family of modules $\mathcal F$.
The main result of this section is the construction of a canonical basis of $\dotU$
which is simultaneously compatible with the canonical bases of the $\mathcal F$.
This basis is a generalization of the canonical basis of $\U^-$ which was constructed in \cite{CHW2}.

In Section 5, we present a construction of a bilinear form on $\dotU$
with nice properties. We further show that this bilinear form can be viewed
as a limit of suitable bilinear forms on modules, which has the consequence
of showing that the canonical basis of $\dotU$ is $\pi$-almost-orthonormal.

Finally, in Section 6 we discuss some extensions of the results in \cite{CFLW}.
In particular, we show that any module in $\mathcal F$ admits a twistor map. 
We then show that twistors preserve the canonical bases of $\dotU$ and of
modules in $\mathcal F$ up to a power of $\sqrt{-1}$.
\vspace{1em}

\noindent\textbf{Acknowledgements:}
I am indebted to Weiqiang Wang, 
as this paper would be non-existent without his guidance.
I would like to thank David Hill, Yiqiang Li, and Zhaobing Fan for
the fruitful collaborations on related projects. I thank
Aaron Lauda for many helpful comments and stimulating conversations;
in particular, for pointing out that the bilinear form
in \cite{CW} has a defect with respect to categorification.

\section{The quantum covering groups}

In this section, we will recall the essential definitions 
and results from earlier papers.
We caution the reader that while the results are largely restatements of
results in \cite{CHW1,CHW2},
several of our conventions differ from those in {\em loc. cit.}
We will also present some new results
which will be of use later on; specifically,
Lemmas \ref{lem:daggerU}, \ref{lem:omega twist cyclicity}, 
and \ref{lem:extra CB props}.

\subsection{The Cartan datum}\label{subsec:datum}

\begin{definition}
  \label{definition:scd}
A {\em Cartan datum} is a pair $(I,\cdot)$ consisting of a finite
set $I$ and a $\Z$-valued symmetric bilinear form $\nu,\nu'\mapsto \nu\cdot\nu'$
on the free abelian group $\Z[I]$  satisfying
\begin{enumerate}
 \item[(a)] $d_i=\frac{i\cdot i}{2}\in \Z_{>0}, \quad \forall i\in I$;

  \item[(b)]
$a_{ij}=2\frac{i\cdot j}{i\cdot i}\in \Z_{\leq 0}$,  for $i\neq j$ in $I$.
\end{enumerate}
A Cartan datum is called an {\em anisotropic super Cartan datum} (or anisotropic datum, in brief) if
there is a partition $I=I_\zero\coprod I_\one$ which satisfies the {\em anisotropic} condition
\begin{enumerate}
        \item[(c)] $2\frac{i\cdot j}{i\cdot i} \in 2\Z$ if $i\in I_\one$.
\end{enumerate}
The $i\in I_\zero$ are called even, $i\in I_\one$ are called odd. We
define a parity function $p:I\rightarrow\set{0,1}$ so that $i\in
I_{\overline{ p(i)}}$.
An anisotropic datum is called  {\em bar-consistent} 
if it additionally satisfies  
\begin{enumerate}
        \item[(d)]  $d_i\equiv p(i) \mod 2, \quad \forall i\in I.$
\end{enumerate} 
\end{definition}
We will always assume an anisotropic Cartan datum satisfies $I_\one\neq\emptyset$ and is bar-consistent in this paper.
We note that a bar-consistent anisotropic datum satisfies
\begin{equation}  \label{eq:even}
i\cdot j\in 2\Z \quad \text{ for all }i,j\in I.
\end{equation}

Let $\nu=\sum_{i\in I} \nu_i i\in \Z[I]$.
We extend the parity function on $I$ to the $\Z$-module homomorphism
$p:\Z[I]\rightarrow \Z_2$ by \[p(\nu)=\sum_{i\in I} \nu_i p(i).\] 
We define the height function $\height$ on $\Z[I]$ 
by letting \[\height(\nu)=\sum_{i\in I} \nu_i.\]
We define the notation 
\begin{equation}\label{eq:deftilderoot}
\tilde{\nu}=\sum d_i\nu_i i.
\end{equation}

A {\em root datum} associated to a anisotropic datum $(I,\cdot)$
consists of
\begin{enumerate}
\item[(a)]
two finitely generated free abelian groups $Y$, $X$ and a
perfect bilinear pairing $\ang{\cdot, \cdot}:Y\times X\rightarrow \Z$;

\item[(b)]
an embedding $I\subset X$ ($i\mapsto i'$) and an embedding $I\subset
Y$ ($i\mapsto i$) satisfying

\item[(c)] $\ang{i,j'}=\frac{2 i\cdot j}{i\cdot i}$ for all $i,j\in I$.
\end{enumerate}

If the image of the imbedding $I\subset X$
(respectively, the image of the imbedding $I\subset Y$) is linearly
independent in $X$ (respectively, in $Y$), then we
say that the root datum is $X$-regular (resp. $Y$-regular).
In this paper, we will assume the root datum is always $Y$-regular.

\begin{remark}
For many results, the $Y$-regular condition can be relaxed using 
similar arguments to those in \cite[Part IV]{L93}.
\end{remark}

If $V$ is a vector space graded by $\Z[I]$, $X$, or $Y$, we will use the weight notation
$|x|=\mu$ if $x\in V_\mu$. If $V$ is a $\Z_2$-graded vector space, we will
use the parity notation $p(x)=a$ if $x\in V_a$.

\subsection{Parameters}

Let $v$ be a formal parameter and let $\pi$ be an indeterminate
such that 
$$
\pi^2=1.
$$ 
For a commutative ring $R$ with $1$, we will form a new ring
$R^\pi=R[\pi]/(\pi^2-1)$. Given an $R^\pi$-module (or algebra) $M$,
the {\em specialization of $M$ at $\pi=\pm 1$} means the $R$-module
(or algebra) $M|_{\pi=\pm 1} =R_{\pm}\otimes_{R^\pi} M$, 
where $R_\pm =R$ is viewed
as a $R^\pi$-module on which $\pi$ acts as $\pm 1$.

Assume 2 is invertible in $R$; i.e. $\frac{1}{2}\in R$. 
We define 
\begin{equation}\label{eq:pi idempotent}
\var_{+}=\frac{1+ \pi}{2},\qquad\var_{-}=\frac{1- \pi}{2},
\end{equation} 
and note that $R^\pi=R\var_+\oplus R\var_-$.
In particular, since $\pi \var_{\pm}=\pm \var_{\pm}$ 
for an $R^\pi$-module $M$, we see that 
\[M|_{\pi=\pm 1}\cong \var_{\pm} M.\]

The principal rings of concern in this paper are $\Qvp$ and $\A=\Zpvvi$. For $k \in \Z_{\ge 0}$ and $n\in \Z$,
we use a $(v,\pi)$-variant of quantum integers, quantum factorial and quantum binomial coefficients:

\begin{equation}
 \label{eq:nvpi}
\begin{split}
\bra{n}_{v,\pi} & 
=\frac{(\pi v)^n-v^{-n}}{\pi v-v^{-1}}  \in \A, 
  \\
\bra{n}_{v,\pi}^!  &= \prod_{l=1}^n \bra{l}_{v,\pi}   \in \A, 
 \\
\bbinom{n}{k}_{v,\pi}
&=\frac{\prod_{l=n-k+1}^n  \big( (\pi v)^{l} -v^{-l} \big)}{\prod_{m=1}^k \big( (\pi v)^{m}- v^{-m} \big)}  \in \A. 
\end{split}
\end{equation}

We will use the notation 
$$
v_i=v^{d_i}, \quad \pi_i=\pi^{d_i}, \quad \text{ for } i\in I.
$$
More generally, for $\nu=\sum \nu_i i$, we set
\[v_\nu=\prod_{i\in I} v_i^{\nu_i},\quad \pi_\nu=\prod_{i\in I} \pi_i^{\nu_i}.\]

There are two important automorphisms of $\Qvp$ and $\A$ for our purposes. 
The {\em bar involution} on $\Qvp$
is the $\Q^\pi$-algebra automorphism defined by $\bar{f(v,\pi)}=f(\pi v^{-1},\pi)$ for $f(v,\pi)\in \Qvp$.
The {\em dagger involution} on $\Qvp$ is the $\Q^\pi$-algebra automorphism defined by 
$f(v,\pi)^\dagger=f(\pi v,\pi)$ for $f(v,\pi)\in \Qvp$.
We note that both involutions restict to $\Zp$-algebra automorphisms of $\A$.

Observe that the $(v,\pi)$-integers are bar-invariant, but are not $\dagger$-invariant in general;
to wit,
\[[k]_{v,\pi}^\dagger=\pi^{k-1}[k]_{v,\pi},\quad
([k]_{v,\pi}^{!})^\dagger=\pi^{\binom{k}{2}}[k]_{v,\pi}^!,\quad
\bbinom{n}{k}_{v,\pi}^\dagger=\pi^{k(n-k)} \bbinom{n}{k}_{v,\pi}.
\]
In particular, we note that $\bbinom{n}{k}_{v,\pi}$ is $\dagger$-invariant
if and only if $n$ is odd or $k$ is even.
\begin{remark}
We note that $[n]^\dagger=\frac{v^n - (\pi v)^{-n}}{v-\pi v^{-1}}$ 
could well be the definition of $(v,\pi)$-integers, and we regard it
as an alternate convention.
\end{remark}

\subsection{The quantum covering groups}\label{subsec:cqg}

For the rest of the paper, we fix an anisotropic datum $(I,\cdot)$ 
and associated root datum $(Y,X,\ang{\cdot,\cdot})$.
We recall some definitions from \cite{CHW1}.

\begin{definition}\cite{CHW1}\label{def:hcqg}
The half-quantum covering group $\ff$ associated to the anisotropic datum 
$(I,\cdot)$ is the $\Z[I]$-graded $\Qvp$-algebra
on the generators $\theta_i$ for $i\in I$ with $|\theta_i|=i$, 
satisfying the relations
\begin{equation}\label{eq:thetaserrerel}
\sum_{k=0}^{b_{ij}} (-1)^k\pi^{\binom{k}{2}p(i)+kp(i)p(j)}
\bbinom{b_{ij}}{k}_{v_i,\pi_i}  \theta_i^{b_{ij}-k}\theta_j\theta_i^k=0 
\;\; (i\neq j),
\end{equation}
\end{definition}

We define the divided powers
\[\theta_i^{(n)}=\theta_i^{n}/\bra{n}_{v_i,\pi_i}^!.\]
Let $\fint$ be the $\A$-algebra generated by $\theta_i^{(n)}$ for various $i\in I$, $n\in \N$.

\begin{definition} \cite{CHW1}
 \label{definition:cqg}
The quantum covering group 
$\UU$ associated to the root datum $((I,\cdot),\ Y,\ X,\ \ang{\cdot,\cdot})$ is the $\Q(v)^{\pi}$-algebra with generators
$E_i, F_i$, $K_\mu$, and $J_\mu$, for  $i\in I$ and $\mu\in Y$, subject to the 
relations:
\begin{equation}\label{eq:JKrels}
J_\mu J_\nu=J_{\mu+\nu},\quad K_\mu K_\nu=K_{\mu+\nu},\quad K_0=J_0=J_\nu^2=1,\quad
J_\mu K_\nu=K_\nu J_\mu,
\end{equation} 
\begin{equation}\label{eq:Jweightrels}
J_\mu E_i=\pi^{\ang{\mu,i'}} E_i J_\mu,\quad J_\mu F_i=\pi^{-\ang{\mu,i'}} F_i J_\mu,
\end{equation} 
\begin{equation}\label{eq:Kweightrels}
K_\mu E_i=v^{\ang{\mu,i'}} E_i K_\mu,\quad K_\mu F_i=v^{-\ang{\mu,i'}} F_i K_\mu,
\end{equation} 
\begin{equation}\label{eq:commutatorrelation}
E_iF_j-\pi^{p(i)p(j)}F_jE_i=\delta_{ij}\frac{J_{d_i i}K_{d_i i}-K_{-d_i i}}{\pi_i v_i- v_i^{-1}},
\end{equation}  
\begin{equation}\label{eq:Eserrerel}
\sum_{k=0}^{b_{ij}} (-1)^k\pi^{\binom{k}{2}p(i)+kp(i)p(j)}\bbinom{b_{ij}}{k}_{v_i,\pi_i} 
 E_i^{b_{ij}-k}E_jE_i^k=0 \;\; (i\neq j),
\end{equation} 
\begin{equation}\label{eq:Fserrerel}
\sum_{k=0}^{b_{ij}} (-1)^k\pi^{\binom{k}{2}p(i)+kp(i)p(j)}\bbinom{b_{ij}}{k}_{v_i,\pi_i} 
 F_i^{b_{ij}-k}F_jF_i^k=0 \;\; (i\neq j),
\end{equation} 
for $i,j\in I$ and $\mu,\nu\in Y$. 
\end{definition}
The algebras $\UU$ and $\ff$ are related in the following way. 
Let $\Um$ be the subalgebra generated by $F_i$ with $i\in I$,
$\Up$ be the subalgebra generated by $E_i$ with $i\in I$,
and $\Uz$ be the subalgebra generated by $K_\nu$ and $J_\nu$ for
$\nu\in Y$. There is an isomorphisms $\ff\rightarrow\Um$ 
(resp. $\ff\rightarrow \Up$) defined by $\theta_i\mapsto \theta_i^-=F_i$
(resp. $\theta_i\mapsto \theta_i^+=E_i$). 

\begin{proposition}\label{prop:UUtriang}
There is a triangular decomposition
\[\UU\cong \Um\otimes\Uz\otimes\Up\cong \Up\otimes \Uz\otimes \Um.\]
\end{proposition}

We define the divided powers
\[E_i^{(n)}=(\theta_i^{(n)})^+,\quad F_i^{(n)}=(\theta_i^{(n)})^-,\]
and set $\Uint^\pm=(\fint)^\pm$.
We will also use the shorthand notation
\[\tJ_\nu=J_{\tilde \nu},\quad \tK_{\nu}=K_{\tilde \nu}.\]
Then for $\nu\in Y$, we also have the $\nu$-integers and $\nu$-binomial coefficients
\[[\nu;n]=\frac{\pi_\nu^n v_\nu^n\tJ_\nu \tK_\nu-\tK_\nu^{-1}v_\nu^{-n}}{\pi_\nu v_\nu-v_\nu^{-1}},\quad \bbinom{\nu; n}{k}=\frac{\prod_{s=1}^k [\nu; n+1-k]}{[k]_{v_\nu,\pi_\nu}^!}.\]
We let $\Uint$ be the $\A$-subalgebra of $\UU$ generated by $E_i^{(n)}$, $F_i^{(n)}$, $J_\nu$, $K_\nu$,
and $\bbinom{i; n}{k}$ for $i\in I$, $\nu\in Y$, $n\geq a\in \N$.

The algebra $\UU$ has a number of important gradings.
We endow $\UU$ with a $\Z[I]$-grading by setting 
\begin{equation}
|E_i|=i,\quad |F_i|=-i,\quad |J_\mu|=|K_\mu|=0,
\end{equation}
and also endow $\UU$ with a $\Z_2$-grading by setting 
\begin{equation}
p(E_i)=p(F_i)=p(i),\quad p(J_\mu)=p(K_\mu)=0.
\end{equation}
We set $\UU_\nu=\set{x\in \UU: |x|=\nu}$.
Note that $p(x)=p(\nu)$ for all $x\in \UU_\nu$.

The algebra $\UU$ has a number of important automorphisms, which we will now recall.
There is a $\Qvp$-algebra automorphism $\omega:\UU\rightarrow\UU$
defined by
\begin{equation}\label{eq:omegadef}
\omega(E_i)=F_i,\quad
\omega(F_i)=\pi_i\tJ_iE_i,\quad
\omega(K_\nu)=K_{-\nu},\quad
\omega(J_\nu)=J_\nu.
\end{equation}
(Note that $\omega$ is not an involution, and our $\omega$ corresponds to $\omega^{-1}$ in the notation of \cite{CHW1}).
For any $\UU$-module $M$, we let ${}^\omega M$ be the $\omega$-twisted
module; that is, the space $M$ equipped with a new action $\cdot$ defined by 
$u\cdot m=\omega(u)m$. We also extend this notation to powers of $\omega$.

There is also an important anti-automorphism of $\UU$. To wit, there is a $\Qvp$-linear map $\rho:\UU\rightarrow\UU$
such that
\begin{equation}\label{eq:rhodef}
\rho(E_i)=\pi_i\tJ_iE_i,\quad
\rho(F_i)=F_i,\quad
\rho(K_\nu)=K_{-\nu},\quad
\rho(J_\nu)=J_\nu,
\end{equation}
and satisfying
\[\rho(xy)=\rho(y)\rho(x).\]

Finally, the bar-involution on $\UU$ is the $\Q^\pi$-algebra automorphism defined by 
\begin{equation}\label{eq:bardef}
\bar E_i=E_i,\quad
\bar F_i=F_i,\quad
\bar K_\nu=J_\nu K_{-\nu},\quad
\bar J_\nu=J_\nu,\quad
\bar v= \pi v^{-1}.
\end{equation}
The maps $\omega$, $\rho$, and $\bar{\phantom{x}}$ (or variations thereof) were defined in \cite{CHW1}. 

We also have a $\Q^\pi$-linear automorphism $\dagger$ on $\UU$ extending 
$\dagger:\Qvp\rightarrow \Qvp$.
\begin{lemma}\label{lem:daggerU}
There exists a $\Q^\pi$-algebra automorphism $\dagger:\U\rightarrow \U$, denoted by
$\bullet\mapsto \bullet^\dagger$, such that
\begin{equation}\label{eq:daggerdef}
E_i^\dagger=\pi_i\tJ_iE_i,\quad
F_i^\dagger=F_i,\quad
K_\nu^\dagger=J_\nu K_{\nu},\quad
J_\nu^\dagger=J_\nu,\quad
v^\dagger= \pi v.
\end{equation}
\end{lemma}
\begin{proof}
To see that $\dagger$ is a well-defined map, we may check that the images of the 
generators satisfy the defining relations. 
All of the relations are trivial to verify except
the Serre relations when $p(i)\neq 0$.
However, since $b_{ij}=1-a_{ij}$ is odd in this case, the binomial coefficients
are dagger-invariant. Then \eqref{eq:Fserrerel} is dagger-invariant, and the image of \eqref{eq:Eserrerel}
is proportional to itself (by a factor of $\pi^{b_{ij}p(i) + p(j)}\tJ_{b_{ij}i+j}$).
\end{proof}

\begin{remark}\label{remark:altdefquantgp}
One interpretation of the automorphism $\dagger$ is as follows. There is an algebra $\UU^\dagger$ 
with generators $E_i^\dagger, F_i^\dagger, J_\nu^\dagger, K_\nu^\dagger$ ($i\in I, \nu\in Y$),
and subject to \eqref{eq:JKrels}-\eqref{eq:Fserrerel} (with $E_i, F_i, J_\nu, K_\nu$ replaced by $E_i^\dagger, F_i^\dagger, J_\nu^\dagger, K_\nu^\dagger$)
except with \eqref{eq:commutatorrelation} replaced by
\[E_i^\dagger F_j^\dagger-\pi^{p(i)p(j)} F_j^\dagger E_i^\dagger=\delta_{ij}\frac{\tK_i^\dagger-\tJ_i^\dagger \tK_{-i}^\dagger}{v_i-\pi_i v_i^{-1}}.\]
The algebra $\UU^\dagger$ may be thought of as $\UU$ defined with respect to the alternate convention
of $(v,\pi)$-integers, and $\dagger$ defines a $\Q^\pi$-algebra isomorphism $\dagger:\UU\rightarrow \UU^\dagger$ defined by
\[E_i\mapsto E_i^\dagger,\quad F_i\mapsto F_i^\dagger,\quad K_\nu\mapsto J_\nu^\dagger K_\nu^\dagger,\quad J_\nu\mapsto J_\nu^\dagger,\quad v\mapsto \pi v.\]
This map will extend naturally to the modified form, but it is unclear what meaning $\dagger$ has in the categorification in \cite{EL}.
\end{remark}

\subsection{Hopf structure}\label{subsec:Hopf}

As in the case of classical quantized enveloping algebras,
we may endow $\UU$ with a Hopf covering algebra structure
in several ways (cf. \cite{CHW1} for details on the counit and antipode). 
As different coproducts have been considered in the papers \cite{CHW1,CHW2}
(see Remark \ref{rem:coproducts} below), we will describe 
some of the more natural choices for the coproducts. Their relationships
are discussed further in Section \ref{subsec:tensor dotU}.

We have homomorphisms 
\[\Delta_1,\Delta_2,\Delta_3,\Delta_4:\UU\rightarrow\UU\otimes \UU\]
such that
\begin{align*}
\Delta_1(E_i)&=E_i\otimes 1+ \tJ_i\tK_i\otimes E_i
&\Delta_1(F_i)&=F_i\otimes \tK_{-i}+1\otimes F_i\\
\Delta_2(E_i)&=E_i\otimes 1+ \tK_i\otimes E_i
&\Delta_2(F_i)&=F_i\otimes \tK_{-i}+\tJ_i\otimes F_i\\
\Delta_3(E_i)&=E_i\otimes \tK_{-i} + 1 \otimes E_i
&\Delta_3(F_i)&=F_i\otimes 1+\tJ_i\tK_i\otimes F_i\\
\Delta_4(E_i)&=E_i\otimes \tK_{-i}+ \tJ_i\otimes E_i
&\Delta_4(F_i)&=F_i\otimes 1+\tK_i\otimes F_i\\
\Delta_s(K_\mu)&=K_\mu \otimes K_\mu\text{ for }s=1,2,3,4.
&\Delta_s(J_\mu)&=J_\mu\otimes J_\mu\text{ for }s=1,2,3,4.
\end{align*}
It is elementary to verify that for each $s=1,2,3,4$, $\Delta_s(\Uint)\subset \Uint\otimes_{\A}\Uint$.
Given $\UU$-modules $M$ and $M'$, we denote by $M\otimes_s M'$ the space $M\otimes_{\Qqp} M'$
with the $\UU$-modules structure given by $\Delta_s$.

We recall from \cite[\S 3.1]{CHW1} the existence of a unique family of elements
$\Theta_{1,\nu}\in \UU^-_{\nu}\otimes \UU^+_{\nu}$ such 
that the element $\Theta_1=\sum_{\nu\in \N[I]}\Theta_{1,\nu}$ satisfies
\[\Delta_1(u)\Theta_1=\Theta_1\bar\Delta_1(u)\]
for all $u\in \UU$ 
(equality in a suitable completion of $\UU\otimes \UU$).
\begin{proposition}\label{prop:otherRmats}
There exists unique families of
elements 
\[\Theta_{2,\nu}\in \UU^-_{\nu}\otimes (\tJ_\nu\UU^+_{\nu})\]
\[\Theta_{3,\nu}\in \UU^+_{\nu}\otimes \UU^-_{\nu}\]
\[\Theta_{4,\nu}\in (\tJ_\nu\UU^+)_{\nu}\otimes \UU^-_{\nu}\]
for $\nu\in \N[I]$
such that the element 
$\Theta_s=\sum_{\nu \in \N[I]} \Theta_{s,\nu}$
for $s=2,3,4$ satisfies
\[\Delta_s(u)\Theta_s=\Theta_s\bar\Delta_s(u)\]
for all $u\in \UU$
(with equality in a suitable completion of $\UU\otimes \UU$).
Moreover, $\Theta_s\bar\Theta_s=1$.
\end{proposition}

\begin{proof}
It is easy to see that the argument in the proof of \cite[Theorem 3.1.1]{CHW1} extends
to other coproducts after making suitable modifications. Alternatively, this will follow
from Lemma \ref{lem:coprodids} and the corresponding properties of $\Theta_1$.
\end{proof}

In this note, the coproduct of principal concern
will be $\Delta_3$, and we use the shortened notations 
$\Delta=\Delta_3$ and $\otimes=\otimes_3$, and $\Theta=\Theta_3$.

\begin{remark}\label{rem:coproducts}
We note that $\Delta_1$ is the main coproduct considered in \cite{CHW1} and $\Delta_4$
is the main coproduct considered in \cite{CHW2}. The coproduct $\Delta_1$ was introduced without specific motivation,
while $\Delta_4$ provides the most natural Tensor Product Rule for crystal bases
(cf. \cite[Theorem 4.12]{CHW2}). We will use $\Delta_3$ since it admits a more natural quasi-$\mathcal R$-matrix and is better
for the construction of canonical bases for the modified quantum group;
compare Theorem \ref{thm:cbdotU} and Corollary \ref{cor:dotCB Delta4 action}.
\end{remark}

\subsection{Highest weight modules and canonical bases}

Recall that a weight module for $\UU$ is a $\UU$-module $M$ satisfying 
\[M=\bigoplus M_\lambda,\quad M_\lambda=\set{m\in M \mid J_\mu K_\nu m=\pi^{\ang{\mu,\lambda}}v^{\ang{\nu,\lambda}} m, \mu,\nu\in Y}.\]
We can define integrable and highest weight modules as usual,
and set $\catO$ be the category of weight modules which are locally $\U^+$-finite.; cf. \cite{CHW1} for more details.

Recall that for $\lambda\in X^+$, the $\UU$-module $M(\lambda)$
is defined to be the space $\ff$ with the $\UU$-action satisfying
\[E_i 1=0,\quad F_i x=\theta_i x,\quad J_\nu K_\mu  x=\pi^{\ang{\nu,|x|'}}v^{\ang{\mu,|x|'}}x.\]
We call $M(\lambda)$ the Verma module.
There is a submodule $\mathcal T$ of $M(\lambda)$ generated by $\theta_i^{\ang{i,\lambda}+1}$
for all $i\in I$; we call the quotient $V(\lambda)=M(\lambda)/\mathcal T$ 
the irreducible module of highest weight $\lambda$,
and denote the image of $1$ in this quotient by $\eta_\lambda$.
\begin{remark} 
In fact, $V(\lambda)$ is not irreducible as a $\UU$-module, as it 
is the sum $V(\lambda)|_{\pi=1}\oplus V(\lambda)_{\pi=-1}$
of irreducible $\UU$-modules (cf. \cite[\S 2.7]{CHW1}).
However, if we define a {\em $\pi$-free weight module} as a weight module where the 
weight spaces are free $\Qvp$-modules, then $V(\lambda)$ is an irreducible
module in the category of $\pi$-free weight modules.
\end{remark}

In the following, we also will often consider the module ${}^\omega V(\lambda)$, which is called the irreducible module
of lowest weight $-\lambda$. We denote the image of $1$ in ${}^\omega V(\lambda)$ by $\xi_{-\lambda}$. Though $\omega$ is not an involution, we note that the 
twisting the action on $V(\lambda)$ by $\omega^2$ does not yield a new module.
To wit, we have the following lemma.

\begin{lemma}\label{lem:omega twist cyclicity}
The $\Qvp$-linear isomorphism
$\omega^2:V(\lambda) \rightarrow {}^{\omega^{2}}V(\lambda)$
given by \[\omega^2(x)= \pi_\nu \pi^{\ang{\tilde\nu,\lambda}} x,\quad x\in V(\lambda)_{\lambda-\nu'},\ \nu\in \N[I],\]
is a $\UU$-module isomorphism.

\end{lemma}
\begin{proof}
For any homogeneous $u\in \UU_\mu$, $\omega^2(u)=\pi_\mu\tJ_\mu u$.  Then for $x\in V(\lambda)_{\lambda-\nu'}$ with $\nu\in \N[I]$, 
we have $\omega^2(ux)=\pi_{\mu+\nu}\pi^{\ang{\tilde{\mu}+\tilde{\nu},\lambda}} ux=\omega^2(u)\omega^2(x)$.
\end{proof}

Since $M(\lambda)=\ff$ as a vector space, it is automatically
endowed with an $\A$-submodule ${}_\A M(\lambda)=\fint$.
We call this the integral form of $M(\lambda)$.
We also have integral forms for the irreducible modules: we call the $\A$-submodule 
${}_\A V(\lambda)=\Uint\xi_{-\lambda}$ (resp. ${}_\A ^\omega V(\lambda)=\Uint\xi_{-\lambda}$)
the integral form of $V(\lambda)$ (resp. ${}^\omega V(\lambda)$).

\subsection{The canonical basis of $\U^-$}
Let us recall some terminology.
Let $R$ be a ring.
A {\em $\pi$-basis} for a free $R^\pi$-module $M$ is a set $S\subset M$
such that there exists an $R^\pi$-basis $B$ for $M$ with $S= B\cup \pi B$.
(We caution the reader that this is called a {\em maximal $\pi$-basis} in \cite{CHW2}.)
We note that a $\pi$-basis of an $R^\pi$-module $M$ is an $R$-basis of $M$.
The fundamental $\pi$-basis in this paper is the following.

\begin{theorem}[\cite{CHW2}]\label{thm:canonicalbasis}
There is a $\pi$-basis $\cB$ of $\ff$ with the following properties:
\begin{enumerate}
\item $\cB$ is a $\pi$-basis of $\fint$ over $\A$.
\item Each $b\in \cB$ is homogeneous.
\item $\bar{b}=b$ for all $b\in \cB$.
\item For $\lambda\in X^+$, there is a subset $\cB(\lambda)$  such that
$\cB(V(\lambda))=\set{b\eta_\lambda:b\in \cB(\lambda)}$ is a $\pi$-basis of $V(\lambda)$, and if $b\in \cB\setminus \cB(\lambda)$, 
$b^- \eta_\lambda=0$.
\end{enumerate}
\end{theorem}

We note that $\cB|_{\pi=1}\subset\ff|_{\pi=1}$ 
is precisely the Lusztig-Kashiwara canonical basis.
We shall need the following additional properties.

\begin{lemma}\label{lem:extra CB props}
Let $(\cdot,\cdot)$ be the polarization on $V(\lambda)$ (cf. \cite{CHW2}).
\begin{enumerate}
\item We have $({}_\A V(\lambda), {}_\A V(\lambda) )\subset \A$.
\item $(b^-\eta_\lambda, b'^-\eta_\lambda)\in {\rm sgn}(b)\delta_{b,b'}+v\Zpv$
for any $b,b'\in\cB$, where $\sgn(b)\in \set{1,\pi}$.
\item For $b\in \cB$, either $\rho(b)=b$ or $\rho(b)=\pi b$, so in particular $\rho(\cB)=\cB$.
\item $b\eta_\lambda=0$ if and only if $b\in \fint \theta_i^{(n)}$
for some $i\in I$ and $n\geq \ang{i,\lambda}+1$.
\end{enumerate}
\end{lemma}

\begin{proof}
(1) and (2) follow from the proof of \cite[Proposition 19.3.3]{L93} and \cite[Proposition 6.1]{CHW2}.

For (3), it is known that $\cB|_{\pi=1}$ is $\rho$-invariant (\cite[Theorem 14.4.3(c)]{L93}).
To prove the more general fact, we utilize the map $\Tw$ in Proposition \ref{prop:halfquantumtwist}; for the convenience of the reader, we will restate
the necessary facts here.
By adjoining $\bt=\sqrt{-1}$ to $\Qvp$ and considering
$\ff[\bt]=\Qvp[\bt]\otimes_{\Qvp} \ff$,
there is a $\Q$-linear map $\Tw:\ff[\bt]\rightarrow \ff[\bt]$
with  $\Tw(\pi)=-\pi$
and $\Tw(b)=\bt^{\ell(b)} b$ for some integer $\ell(b)$.
Note that, in particular, $\Tw(\epsilon_+ b)=\bt^{\ell(b)}\epsilon_- b$.

By \cite[Proposition 2.6]{CFLW}, $\rho\Tw$ and $\Tw\rho$ are equal up to a sign.
Then $\Tw(\rho(b))= \varsigma \bt^{\ell(b)} \rho(b)$, where
$\varsigma=\pm 1$. But then using $\rho$-invariance
of the canonical basis when $\pi=1$,
\[\bt^{\ell(b)} \epsilon_{-}b=\Tw(\epsilon_+ b)=\Tw(\rho(\epsilon_+b))=\varsigma \bt^{\ell(b)}\epsilon_-\rho(b).\]
Since $-\epsilon_-=\pi\epsilon_-$, $\rho(\epsilon_- b)=\epsilon_- b$ or $\epsilon_- \pi b$.
Therefore, $\rho(b)=b$ or $\rho(b)=\pi b$, as claimed.

Finally, we note that $\Tw(\ff[\bt]\theta_i^{n})\subset \ff[\bt]\theta_i^{n}$,
the statement of (4) holds for $\cB|_{\pi=1}$ (cf. \cite[Theorem 14.4.11]{L93}),
and we have $\Tw(\epsilon_\pm b)=\bt^{\ell(b)}\epsilon_\mp b$. Combining these
facts proves (4) for $\cB|_{\pi=-1}$ and hence for $\cB$.

\end{proof}

\section{The modified quantum covering groups}

In this section, we will define the modified quantum group and systematically develop its fundamental properties. We also define a family of modules $\mathcal F$ which
will be important in the subsequent sections.

\subsection{The modified form}

The modified quantum group may be defined as a direct sum of certain quotients of $\UU$ (cf. \cite[Chapter 23]{L93})
but the essential description of this algebra may be given as follows.

\begin{definition}\label{def:mqg}
The {\em modified quantum covering group} $\dotU$ associated to the root datum $(Y,X,I,\cdot)$
is the associative $\Qvp$-algebra without unit on symbols $x1_\lambda$ and $1_\lambda x$
for $x\in \UU$ and $\lambda\in X$ satisfying the relations
\[x1_\lambda y1_\eta=\delta_{\lambda,\eta+|y|} (xy)1_\lambda,\quad
 x1_\lambda=1_{\lambda+|x|} x\qquad 
 \text{for all homogeneous } x,y\in \UU,\lambda\in X,\]
\[K_\nu1_\lambda=v^{\ang{\nu,\lambda}}1_\lambda,\quad
J_\nu1_\lambda=\pi^{\ang{\nu,\lambda}}1_\lambda\qquad 
 \text{for all } \nu\in Y,\lambda\in X.\]
\end{definition}

\begin{remark}
A version of $\dotU$ has already been defined in \cite{CFLW,CW} in different ways. 
In \cite[Definition 4.2]{CFLW}, the modified form is defined using generators $1_\lambda$, 
$E_i1_\lambda$ and $F_i1_\lambda$ for $i\in I$ and $\lambda\in X$
satisfying certain relations; it is straightforward to see that this is equivalent to our definition.
In \cite[\S 6.1]{CW}, a rank one modified form is defined using certain quotients of $\U$ 
in direct parallel to the definition in \cite[\S 23.1]{L93}.
This construction can be generalized to higher rank, and results in an algebra isomorphic to our definition.
\end{remark}

The algebra $\dotU$ is naturally $\UU$-bimodule
under the following action: for $x,y,z\in \UU$,
we set \[x (y1_\lambda)z=(xyz)1_{\lambda-|z|}.\]
We note the following commutation relations which may be deduced from \cite[Lemma 2.1.6]{CHW1} or \cite[Proposition 6.1]{CW}.
\begin{lemma} \label{lem:dotU commutation relations}
We have the following identities in $\dotU$:
\begin{align*}
&E_i^{(N)}1_\lambda F_i^{(M)}=\sum_t \pi_i^{MN-\binom{t+1}{2}} F_i^{(M-t)}
 \bbinom{M+N+\ang{i,\lambda}}{t}_{v_i,\pi_i} 1_{\lambda+2N+2M-2t}E_i^{(N-t)}, \\
&F_i^{(N)}1_\lambda E_i^{(M)}=\sum_t (-1)^t \pi_i^{(M-t)(N-t)-t^2}
 E_i^{(M-t)}\bbinom{M+N-\ang{i,\lambda}}{t}_{v_i,\pi_i}1_{\lambda-2N-2M+2t}F_i^{(N-t)}, \\
&E_i^{(N)} F_j^{(M)}1_\lambda=\pi^{MNp(i)p(j)} F_j^{(M)}E_i^{(N)}1_\lambda\quad\text{if }i\neq j.
\end{align*}
\end{lemma}

The next proposition shows that the
various algebra (anti)automorphisms on $\UU$ we have previously 
defined have analogous (anti)automorphisms on $\dotU$ which are compatible
with the bimodule presentation.
\begin{proposition}\label{prop:dotUmorphisms} \hspace{.1em}

\begin{enumerate} 
\item There exists a $\Qvp$-algebra automorphism $\omega$ of $\dotU$ satisfying \[\omega(x1_\lambda)=\omega(x)1_{-\lambda}.\]
\item There exists a $\Qvp$-algebra antiautomorphism $\rho$ of $\UU$ satisfying \[\rho(x1_\lambda)=1_{-\lambda}\rho(x).\]
\item There exists a $\Q^\pi$-algebra involution $\bar{\phantom{x}}$ of $\dotU$ satisfying \[\bar{x1_\lambda}=\bar{x}1_\lambda.\]
\item  There exists a $\Q^\pi$-algebra involution $\dagger$ of $\dotU$ satisfying \[(x1_\lambda)^\dagger=x^\dagger 1_\lambda.\]
\end{enumerate}
\end{proposition}
\begin{proof}
This follows as an elementary consequence of the existence of these maps on $\UU$
and the definition of $\dotU$.
\end{proof}

A $\dotU$-module $M$ is called {\em unital}
if for any $m\in M$, $m=\bigoplus_{\lambda\in X} 1_\lambda m$ with $1_\lambda m=0$ for all but finitely
many $\lambda\in X$. As observed by Lusztig \cite{L93}, unital $\dotU$-modules may naturally be viewed as 
weight $\UU$-modules by interpreting the idempotent $1_\lambda$ as 
a projection onto the $\lambda$-weight space. 
More precisely, any unital $\dotU$-module $M$ is a weight $\UU$-module under
the action $x\cdot m=\sum_{\lambda\in X} x1_\lambda m$, and similarly any weight $\UU$-module
$M$ is a unital $\dotU$-module via $x1_\lambda m=xm_\lambda$, where $m_\lambda$ is the orthogonal
projection of $m$ to the $M_\lambda$ weight space. 

\subsection{Integral form}
Suppose $M$ and $M'$ are free $\Qvp$-modules with 
$\pi$-bases $S\subset M$ and $S'\subset M'$.
The $S\otimes S=\set{s\otimes s\mid s\in S, s'\in S'}$ is clearly
a $\pi$-basis for $M\otimes M'$. However, we note the map $S\times S'\rightarrow S\otimes S'$ given by $(s,s')\mapsto s\otimes s'$ is not a bijection, since
clearly $s\otimes s'=(\pi s)\otimes (\pi s')$. We will occasionally need
an honest index set, so let $\sim$ be the equivalence relation on $S\times S'$
given by $(\pi s, s')\sim (s,\pi s')$. Then we set
\begin{equation}\label{eq:timespi def}
S\times_\pi S'= S\times S'/\sim.
\end{equation}

From the triangular decomposition of $\UU$, $\UU$ has a $\pi$-basis consisting of elements of the form $b^+J_\mu K_\nu b'^-$
where $(b,b')\in \cB\times_\pi \cB$ and $\mu,\nu\in Y$. Since
\[
(b^+J_\mu K_\nu b'^-)1_\lambda=
\pi^{\ang{\mu,\lambda_1}}v^{\ang{\nu,\lambda_1}}
b^+1_{\lambda_1} b'^-,
\]
where $\lambda_1=\lambda-|b'|$, we see that
$\set{b^+1_\lambda b'^-: (b,b')\in\cB\times_\pi \cB}$ forms a $\pi$-basis of $\dotU$.
Similarly, $\set{b^-1_\lambda b'^+:(b,b')\in\cB\times_\pi \cB}$ forms a $\pi$-basis of $\dotU$.
In fact, these elements span an integral form of $\dotU$.

\begin{lemma} \hspace{.1em}

\begin{enumerate}
\item The $\A$-submodule of $\dotU$ spanned
by the elements $x^+1_\lambda x'^-$ (with $x,x'\in \fint$)
coincides with the $\A$-submodule of $\dotU$
spanned by the elements $x^-1_\lambda x'^+$ (with $x,x'\in \fint$).
We denote it by $\dotUint$.
\item The elements $\set{b^+1_\lambda b'^-: (b,b')\in\cB\times_\pi \cB}$ (resp. 
$\set{b^-1_\lambda b'^+: (b,b')\in\cB\times_\pi \cB}$) form a $\pi$-basis
of $\dotUint$.
\item $\dotUint$ is a $\A$-subalgebra of $\dotU$
which is generated by the elements 
$E_i^{(n)}1_\lambda$ and $F_i^{(n)}1_\lambda$ for $i\in I$, $n\geq 0$,
and $\lambda\in X$.
\end{enumerate}
\end{lemma}

\begin{proof}
Recall that $\fint$ is the $\Z[v^{\pm 1}]^\pi$ 
generated by $\theta_i^{(n)}$ for $i\in I$ and 
$n\geq 0$, and so repeated application of 
Lemma \ref{lem:dotU commutation relations} implies (1) and (3).
(2) follows from the triangular decomposition of $\dotU$.
\end{proof}

\subsection{Tensor modules of $\dotU$}\label{subsec:tensor dotU}
The algebra $\dotU$ does not admit a natural co-product, as any candidate would require
an infinite sum (cf. \cite[\S 23.1.5]{L93}). However,
since unital modules and weight modules are equivalent, the Hopf structure on $\U$
imposes a unital $\dotU$-module structure on the tensor product of weight modules.
In the following, we will occasionally need to consider these tensor products as modules under different
coproducts of $\UU$. To that end, we will now relate the coproducts from \S \ref{subsec:Hopf}.

For $s\in \set{1,2,3,4}$, let $\Delta_s^{\omega}=(\omega^{-1}\otimes \omega^{-1})\circ\Delta_s\circ \omega$
and $\Delta_s^{\omega^{-1}}=(\omega\otimes \omega)\circ\Delta_s\circ \omega^{-1}$.
Let
$\bar{\Delta_s}=\bar{\phantom x}\circ \Delta_s \circ \bar{\phantom x}$,
where $\bar{\phantom x}=\bar{\phantom x}\otimes \bar{\phantom x}:
\UU\otimes \UU\rightarrow \UU\otimes\UU$.
Also, let ${}^\tau\Delta_s$ be the composition
of $\Delta_s$ with the automorphism $\tau:\UU\otimes \UU\rightarrow
\UU\otimes \UU$ given by $x\otimes y\mapsto \pi^{p(x)p(y)} y\otimes x$.
In particular, we note that ${}^\tau\bar\Delta_{s}^{\omega}$
is a well defined symbol, since $\tau$, $\barmap\otimes\barmap$ and $\omega\otimes \omega$ all commute,
and $\omega$ and $\bar{\phantom{x}}$ commute.
\begin{lemma}\label{lem:coprodids}
We have the following identifications:
\[\Delta_2={}^\tau \Delta_1^{\omega}\]
\[\Delta_3={}^\tau\bar\Delta_1\]
\[\Delta_4={}^\tau\Delta_3^{\omega^{-1}}=\bar{\Delta}_1^{\omega^{-1}}\]
\end{lemma}

\begin{proof}
First, let us assume all the equalities but ${}^\tau\Delta_3^{\omega^{-1}}=\bar{\Delta}_1^{\omega^{-1}}$.
Then we see that 
\[{}^\tau\Delta_3^{\omega}=\tau\circ\omega\otimes \omega\circ\tau\circ\bar{\phantom{x}}\circ\Delta_1\circ\bar{\phantom{x}}\circ\omega^{-1}.\]
Since $\tau$ is an involution commuting with $\bar{\phantom{x}}$ and $\omega\otimes \omega$,
we see the desired equality holds.

It remains to check the other equalities.
Since all of the maps on the right-hand side are compositions of algebra homomorphisms, 
we may simply check on generators. The lemma is clear for $K_\mu$ and $J_\mu$,
so it suffices to check the equalities on $E_i$ and $F_i$.

We compute
\begin{align*}
\Delta_1^{\omega}(E_i)
&=E_i\otimes \tK_i + 1\otimes E_i,
&\Delta_1^\omega(F_i)
&=F_i\otimes \tJ_i + \tK_{-i} \otimes F_i,
\\
\bar{\Delta}_1(E_i)
&=E_i\otimes 1 + \tK_{-i}\otimes E_i,
&\bar \Delta_1(F_i)
&=F_i\otimes \tJ_i \tK_i + 1 \otimes F_i,
\\
\Delta_3^{\omega^{-1}}(E_i)
&=E_i\otimes \tJ_i + \tK_{-i}\otimes E_i,
&\Delta_3^{\omega^{-1}}(F_i)
&=F_i\otimes \tK_i + 1 \otimes F_i.
\end{align*}
Applying $\tau$ to these identities completes the proof of the lemma.

\end{proof}

We note the following consequence of the lemma.

\begin{lemma}\label{cor:comparing twists}
For any $\UU$-modules $M,M'$, there is a $\UU$-module isomorphism 
\[\tau: {}^{\omega}(M\otimes_4 M')\rightarrow {}^{\omega} M'\otimes {}^{\omega}M\]
given by $\tau(x\otimes y): \pi^{p(y)p(x)}y\otimes x$.
\end{lemma}
\begin{proof}
Let $u\in \UU$ and suppose $\Delta_4(\omega(u))=\sum u_1\otimes u_2$ with $u_1,u_2$ homogeneous.
Then 
\begin{align*}
\tau(\Delta_4(\omega(u)) (x\otimes y))
&=\tau(\sum \pi^{p(x)p(u_2)}u_1 x\otimes u_2 y)\\
&=\sum \pi^{p(x)p(u_2)+p(u_1x)p(u_2y)}\sum  u_2y\otimes u_1x\\
&=(\sum \pi^{p(u_2)p(u_1)} \omega^{-1}(u_2)\otimes \omega^{-1}(u_1)) \tau(x\otimes y)\\
&=\tau\circ(\omega^{-1} \otimes \omega^{-1})\circ \Delta_4(\omega(u)) \tau(x\otimes y)\\
&=\Delta(u)\tau(x\otimes y).
\end{align*}
\end{proof}

\subsection{The Verma identification}

Recall that the Verma module $M(\lambda)$ is isomorphic to $\ff$ as a $\Qvp$-module.
We shall now describe a generalization of this identification to $\dotU$.

Let $\lambda,\lambda'\in X$. As discussed in \S\ref{subsec:tensor dotU}, 
the $\UU$-module $M(\lambda)\otimes {}^\omega M(\lambda')$ is naturally a 
$\dotU$-module. As it will often be convenient, 
we will use the $\Qvp$-module identification 
\[ M(\lambda)\otimes {}^\omega M(\lambda')=\ff\otimes\ff.\]
From the triangular decomposition of $\dotU$, the
elements $b^+b'^-1_{\zeta}$ with $(b,b')\in\cB\times_\pi \cB$ comprise a $\pi$-basis of $\dotU$.
Likewise, the canonical basis of $\U^-\cong \ff$ induces a $\pi$-basis
$\cB\otimes \cB=\set{b\otimes b'\mid (b,b')\in \cB\times_\pi \cB}$ 
on $\ff\otimes \ff$.

If $\zeta=\lambda-\lambda'$, then since a $\pi$-basis is a $\Q(v)$-basis, we have
\begin{equation}\label{eq:+-tensor}
b^+b'^-1_\zeta(1\otimes 1)
=b^+(b'\otimes 1)
=\pi^{p(b')p(b)}
b'\otimes b+\sum c_{b_1',b_1} b_1'\otimes b_1,
\end{equation}
where $c_{b_1',b_1}\in \Q(v)$ are constants, $\nu=|b|$,
and the sum is over $(b_1,b_1')\in\cB\times_\pi \cB$
such that $\height |b_1|<\height |b|$, $\height |b_1'| < \height |b'|$
and $b_1'$ belongs to the $\dotU$-submodule of $M(\lambda)$ generated
by $b'$. Moreover, since $\cB\subset \Uint$, the left-hand side of \eqref{eq:+-tensor}
lies in $\displaystyle _\A\ff\otimes_{\A} {}_\A\ff$ and hence
we have $c_{b_1',b_1}\in \Z[v,v^{-1}]$.

Similarly, from the triangular decomposition of $\dotU$, the
elements $b^-b'^+1_{\zeta}$ comprise a $\pi$-basis of $\dotU$.
If $\zeta=\lambda-\lambda'$, then
\begin{equation}\label{eq:-+tensor}
b^-b'^+1_\zeta(1\otimes 1)
=b^-(1\otimes b')
=b\otimes b'+\sum c_{b_1,b_1'} b_1\otimes b_1',
\end{equation}
where $c'_{b_1,b_1'}\in \Z[v, v^{-1}]$ are constants and the sum is over $b_1,b_1'\in \cB\times_\pi \cB$
such that $\height |b_1|<\height |b|$, $\height |b_1'| < \height |b'|$
and $b_1'$ belongs to the $\dotU$-submodule of 
${}^\omega M(\lambda')$ generated by $b'$.

In either case, note that the transition matrix between the $\pi$-basis
$\cB\otimes \cB$ and the elements $\set{b^-b'^+1_\zeta(1\otimes 1)\mid b,b'\in \cB\times_\pi \cB}$ 
is upper unitriangular with entries in $\Z[v,v^{-1}]$, 
hence the latter is also a $\pi$-basis.
In particular, we see that the $\Qvp$-linear map 
\begin{equation}
\eth_{\lambda,\lambda'}:\dotU 1_{\lambda-\lambda'}\rightarrow M(\lambda)\otimes {}^\omega M(\lambda'),
\quad u\mapsto u(1\otimes 1)
\end{equation}
is an isomorphism. Similarly, the $\A$-linear map

\begin{equation}
{}_\A\eth_{\lambda,\lambda'}:\dotUint 1_{\lambda-\lambda'}\rightarrow
{}_\A M(\lambda)\otimes_\A {}^\omega_\A M(\lambda'),
\quad u\mapsto u(1\otimes 1)
\end{equation}
is an isomorphism.

\subsection{The family $\mathcal F$}\label{subsec:dotU and N}

The following particular family of $\UU$-modules will be of critical importance later on.
Let $\mathcal F=\set{N(\lambda,\lambda')\mid \lambda,\lambda'\in X^+}$, where
\[
N(\lambda,\lambda')=V(\lambda)\otimes_{\Qqp} {}^\omega V(\lambda'),\quad 
{}_\A N(\lambda,\lambda')={}_\A V(\lambda)\otimes_{\A}({}_\A^{\omega} V(\lambda')).
\]
Denote by $N_s(\lambda,\lambda')$ (resp. ${}_\A N_s(\lambda,\lambda')$) the $\UU$-module (resp. $\Uint$-module) obtained by
acting on $N(\lambda,\lambda')$ (resp. ${}_\A N(\lambda,\lambda')$) via $\Delta_s$. 
When there can be no confusion, we will shorten notation to $N(\lambda,\lambda')=N_3(\lambda,\lambda')$
and ${}_\A N(\lambda,\lambda')={}_\A N_3(\lambda,\lambda')$.

The module structures on $N_3(\lambda,\lambda')$ and $N_4(\lambda,\lambda')$
are quite closely related, as demonstrated by the following lemma.

\begin{lemma}\label{lem:N34iso}
The linear isomorphism $N_3(\lambda,\lambda')\rightarrow N_4(\lambda,\lambda')$
given by 
\[
x\otimes y\mapsto \pi^{\ang{\tilde\nu,\lambda}} x\otimes y\quad \text{ for all }x\in V(\lambda),\ y\in {}^\omega V(\lambda')_{\nu'-\lambda'},
\ \nu\in \N[I],
\]
is a $\UU$-module isomorphism.
\end{lemma}

\begin{proof}
This is elementary to verify using the definitions.
\end{proof}
Now let us examine the relationship between $N(\lambda,\lambda')$ and $M(\lambda)\otimes {}^\omega M(\lambda')$.
Let $\zeta\in X$ and let $a=\sum_i a_i i, a'=\sum_i a'_i i\in \N[I]$
such that $\ang{i,\zeta}=a_i'-a_i$ for all $i\in I$.
We define the ideals 
\[
P(\zeta,a,a')=\sum_{i,n>a_i} \dotU F_i^{(n)} 1_\zeta+
\sum_{i,n>a_i'} \dotU E_i^{(n)} 1_\zeta,
\] 
\[
{}_\A P(\zeta,a,a')=\sum_{i,n>a_i} \dotUint F_i^{(n)} 1_\zeta+
\sum_{i,n>a_i'} \dotUint E_i^{(n)} 1_\zeta.
\]

Let $\lambda,\eta\in X^+$
and set $a_i=\ang{i,\lambda}$, $a_i'=\ang{i,\eta}$ for all $i\in I$.
Set $\xi=\eta-\lambda$, $a=\sum_i a_i i\in \N[I]$, 
and $a'=\sum_i a_i'\in \N[I]$. Let $\mathcal T$ (resp. $\mathcal T'$)
be the kernel of the canonical homomorphism of $\UU$-modules
$\ff=M(\lambda)\rightarrow  V(\lambda)$
(resp. $\ff={}^\omega M(\lambda')\rightarrow  {}^\omega V(\lambda')$). Then by 
Theorem \ref{thm:canonicalbasis} (4) and Lemma \ref{lem:extra CB props} (4), $\mathcal T$ 
(resp. $\mathcal T'$) is generated by $b\in\cB$ such that $b\in\fint \theta_i^{(n)}$
for some $n\geq a_i$ (resp. $n\geq a_i'$).

Then taking tensor products, we obtain the surjective homomorphism
\[
M(\lambda)\otimes {}^\omega M(\lambda')\rightarrow 
 N(\lambda,\lambda').
\]
The kernel of this map is the subspace 
$\mathcal T\otimes \ff+\ff\otimes\mathcal T'$. 

Let $\mathcal D(\lambda)=\cB\setminus \cB(\lambda)$.
Now by the triangular decomposition, the description of $\mathcal T$, and \eqref{eq:+-tensor},
$\eth_{\lambda,\lambda'}$ maps the subspace 
\[\sum_{ b\in D(\lambda), b'\in\cB} \Qvp b'^+b^-1_\zeta=\sum_{i,n>a_i} \dotU F_i^{(n)} 1_\zeta\]
onto the subspace
\[\sum_{b\in D(\lambda), b'\in \cB} \Qvp b\otimes b'=\mathcal T\otimes\ff.\]
Similarly by the triangular decomposition, the description of $\mathcal T'$, and \eqref{eq:-+tensor},
$\eth_{\lambda,\lambda'}$ maps the subspace 
\[\sum_{b\in\cB, b'\in D(\lambda')} \Qvp b^-b'^+1_\zeta=\sum_{i,n>a_i} \dotU E_i^{(n)} 1_\zeta\]
onto the subspace
\[\sum_{b\in \cB,b'\in D(\lambda')} \Qvp b\otimes b'=\ff\otimes\mathcal T'.\]
We note that replacing everything with the integral form in the preceding argument
does not effect the argument, and thus we have proven the following.

\begin{proposition}\label{prop:dotU onto N}
The map $u\mapsto u(\eta_\lambda\otimes \xi_{-\lambda'})$ defines a surjective linear map
$\dotU\rightarrow N(\lambda,\lambda')$ with kernel equal to $P(\zeta,a,a')$. Moreover,
restriction of this map to $\dotUint$ gives a surjective linear map $\dotUint\rightarrow {}_\A N(\lambda,\lambda')$
with kernel equal to ${}_\A P(\zeta,a,a')$.
\end{proposition}

We now note that each module in $\mathcal F$ comes equipped with a $\pi$-basis.
Indeed, the elements 
\[\cB(\lambda,\lambda')=\set{b^-\eta_{\lambda}\otimes b^+\xi_{-\lambda'}:b\times b'\in B(\lambda)\times_\pi B(\lambda')}\]
form a $\pi$-basis of $N(\lambda,\lambda')$.
They generate a $\Zpv$-submodule $\cL=\cL(\lambda,\lambda')$ and a $\A$-submodule ${}_\A \cL$.

\section{Canonical bases}

In this section, we will construct bar-invariant canonical bases for the modules
in $\mathcal F$ and show that they are mutually compatible with respect
to a family of maps. This enables us to construct a ``projective limit''
of these canonical bases, and this limit is a canonical basis for $\dotU$.

\subsection{Bar involution for $\mathcal F$}\label{subsec:R-matrix integrality}

Recall the quasi-R-matrix  $\Theta=\sum_{\nu\in \N[I]} \Theta_\nu$ 
from Proposition \ref{prop:otherRmats}. We recall in particular that $\Theta$
lives in a suitable completion of $\UU\otimes \UU$, and that
(in this completion) we have
\[\Delta(u)\Theta=\Theta\bar \Delta(u),\quad \Theta\bar\Theta=1.\]
where $\bar{\phantom{x}}:\UU\otimes \UU\rightarrow\UU\otimes \UU$ is the map
$\bar{x\otimes y}=\bar x\otimes \bar{y}$.

Let $M$ and $M'$ be weight modules such that ${}^\omega M\in \catO$ or $M'\in \catO$.
Then since $\Theta_\nu\in \U^+_\nu\otimes\U_\nu^-$, given any $x\in M\otimes M'$,
we must have $\Theta_\nu x=0$ for all but a finite number of $\nu\in \N[I]$. 
Then by regarding $M\otimes M'$ as a $\UU\otimes \UU$-module,
$\Theta$ defines a linear map $\Theta: M\otimes M'\rightarrow M\otimes M'$ by
$m\otimes m'\mapsto \sum \Theta_\nu (m\otimes m')$. This is well-defined because
only finitely many terms may be non-zero, and we see that
\[\Delta(u)\Theta(m\otimes m')=\Theta(\bar{\Delta(\bar u)}m\otimes m').\]
In particular, suppose that $M$ and $M'$ are equipped with bar-involutions
$\bar{\phantom{x}}:M\rightarrow M$ and $\bar{\phantom{x}}:M'\rightarrow M'$
such that $\bar{um}=\bar{u}\bar{m}$ and $\bar{um'}=\bar{u}\bar{m'}$
for all $u\in \UU$, $m\in M$ and $m'\in M'$.
Then 
\[\Delta(u)\Theta(m\otimes m')=\Theta(\bar{\Delta(\bar u)\bar m\otimes \bar m'}),\]
where $\bar{\phantom{x}}=\bar{\phantom x} \otimes \bar{\phantom x}
:M\otimes M'\rightarrow M\otimes M'$.

\begin{theorem}
Let $\lambda,\lambda'\in X$ and consider the Verma modules $M(\lambda)$ 
and $M(\lambda')$. Set $M=M(\lambda)\in \catO$ and
$M'={}^\omega M(\lambda')\in \mathcal C$. Then $\Theta$ is a well defined map on $M\otimes M'$
which leaves stable the $\A$-submodule ${}_\A M(\lambda)\otimes_\A({}^\omega_\A M(\lambda'))$.
\end{theorem}

\begin{proof} 

The proof of this result is essentially the same as in \cite[Prop 24.1.4]{L93},
but we shall state it here for completeness.

Since the ambient space of $ M(\lambda)$ and $M(\lambda')$
is $\ff$, there are well defined maps $\bar{\phantom x}:M\rightarrow M$
and $\bar{\phantom x}: M'\rightarrow M'$. On the other hand, $M(\lambda)$
and ${}^\omega M(\lambda')$ may be identified as $\UU$-modules with certain
quotients of $\UU$ such that the bar involution on $\UU$ induces those on
$M$ and $M'$, whence $\bar{um}=\bar{u}\bar{m}$ and $\bar{um'}=\bar{u}\bar{m'}$
for $u\in \UU$, $m\in M$ and $m'\in M'$. Then we set 
$\bar{\phantom x}=\bar{\phantom x}\otimes \bar{\phantom x}: 
M\otimes M'\rightarrow M\otimes M'$.

Now let us identify $M(\lambda)$ and ${}^\omega M(\lambda')$ with $\ff$.
We note that by definition, $\bar{1}=1$ in $M$ and $M'$.
Moreover, $\Theta(1\otimes 1)=1\otimes 1$. Then we have
\[u(1\otimes 1)=\Theta(\bar{\bar u( 1\otimes 1)}).\]

Since the ambient space of ${}_\A M(\lambda)$
and ${}^\omega_\A M(\lambda')$ is $\fint$, which is bar-invariant, we see that
${}_\A M(\lambda)\otimes_\A
({}^\omega_\A M(\lambda'))$ is stable under $\bar{\phantom x}$.
Take $x\in  {}_\A M(\lambda)\otimes_\A
({}^\omega_\A M(\lambda'))$, and set $x'=\bar{x}\in {}_\A M(\lambda)\otimes_\A
({}^\omega_\A M(\lambda'))$.

On the other hand, the isomorphism 
$\dotUint 1_{\lambda'-\lambda}\rightarrow {}_{\A}M(\lambda)\otimes_\A
({}^\omega_\A M(\lambda'))$
implies there is a $u'\in \dotUint 1_{\lambda-\lambda'}$ such that
$u'(1\otimes 1)=x'$. There is also a $u=\bar{u'}\in \dotUint 1_{\lambda-\lambda'}$.
Therefore, $x=\bar{x'}=\bar{u'(1\otimes 1)}=\bar{\bar{u}(1\otimes 1)}$,
and so $\Theta(x)=u(1\otimes 1)\in {}_\A M(\lambda)\otimes_\A
({}^\omega_\A M(\lambda'))$.

\end{proof}

This immediately implies the following corollary.

\begin{corollary}\label{cor:ThetaIntegrality}
The map $\Theta$ 
leaves stable the $\A$-submodule ${}_\A N(\lambda,\lambda')$ 
(resp. ${}_\A N'(\lambda,\lambda')$)
of $N(\lambda,\lambda')$ (resp. $N'(\lambda,\lambda')$).
\end{corollary}

Let $\bar{\phantom x}:V(\lambda')\rightarrow V(\lambda')$
be the unique $\Qp$-linear involution such that
$\bar{u \eta_{\lambda'}}=\bar u \eta_{\lambda'}$
for all $u\in\UU$; similarly, let
$\bar{\phantom x}:{}^\omega V(\lambda)\rightarrow {}^{\omega}V(\lambda)$
be the unique $\Qp$-linear involution such that
$\bar{u \xi_{-\lambda}}=\bar u \xi_{-\lambda}$
for all $u\in\UU$. Let 
$\bar{\phantom x}=\bar{\phantom x}\otimes \bar{\phantom x}:
N(\lambda,\lambda')\rightarrow 
N(\lambda,\lambda')$.

Since the maps $\Theta,\bar{\phantom x}:
N(\lambda,\lambda')\rightarrow N(\lambda,\lambda')$
are well-defined, let $\Psi=\Theta\circ\bar{\phantom{x}}$.
Then note that
\[\Psi(\Delta(u)m\otimes m')
=\Theta(\bar{\Delta(u)m\otimes m'})
=\Delta(\bar{u}) \Theta(\bar{m}\otimes \bar{m'})
=\Delta(\bar{u})\Psi(m\otimes m'),\]
and that $\Psi^2=1$, from whence we call
$\Psi$ the {\em bar-involution} on $N(\lambda,\lambda')$.

\subsection{General lemma on semi-linear equations}

We will now present an analogue of \cite[\S 24.2]{L93} in the covering setting.

\begin{lemma}\label{lem:semi-linear}
Let $H$ be a set with a partial order $\leq$ such that for any $h\leq h'$ in $H$,
the set $\set{h'':h\leq h''\leq h'}$ is finite. Assume that for each $h\leq h'$
in $H$, there exists an element $r_{h,h'}\in \A$ such that
$r_{h,h}=1$ and 
\[\sum_{h''; h\leq h''\leq h'} \bar r_{h,h''} r_{h'',h'}=\delta_{h,h'}\]
for all $h\leq h'\in H$.

Then there is a unique family of elements $p_{h,h'}\in \Zpv$
defined for all $h\leq h'\in H$ such that $p_{h,h}=1$, $p_{h,h'}\in v\Zpv$
for all $h<h'$ in $H$, and
\[p_{h,h'}=\sum_{h''; h\leq h''\leq h'} \bar p_{h,h''} r_{h'',h'}\]
for all $h\leq h'\in H$.
\end{lemma}

\begin{proof}
For $h\leq h'$ in $H$, denote by $d(h,h')$ the maximum length of a chain
$h=h_0<h_1<\ldots <h_\ell=h'\in H$. Note that $d(h,h')<\infty$ by our assumption
on $H$. For any $n\geq 0$, let $P_n$ be the assertion of the lemma restricted
to those $h\leq h'$ such that $d(h,h')\leq n$ (and note that all the assertions
make sense under this restriction. We will prove $P_n$ by induction.

First note $P_0$ is trivial and assume that $n\geq 1$. Let $h\leq h'$.
If $d(h,h')<n$ then $p_{h,h'}$ is defined by $P_{n-1}$. If $d(h,h')=n$,
then $q=\sum_{h''; h\leq h'' < h'} \bar p_{h,h''} r_{h'',h'}$ is defined.

First, we shall show that $q+\bar{q}=0$.
Indeed, using $P_{n-1}$ and the assumptions of the lemma,
\begin{align*}
q+\bar{q}&=\sum_{h_1; h\leq h_0< h'} \bar p_{h,h_0} r_{h_0,h'}
+\sum_{h_1; h\leq h _1 < h'} p_{h,h_1} \bar r_{h_1,h'}
\\&=\sum_{h_0,h_1; h\leq h_0 <h_1=h'} \bar p_{h,h_0} r_{h_0,h_1}\bar r_{h_1,h'}
+\sum_{h_0, h_1; h\leq h_0\leq h_1< h'} \bar p_{h,h_0} r_{h_0,h_1} \bar r_{h_1,h'}\\
&=\sum_{h_0, h_1; h\leq h_0\leq h_1\leq  h';h_0<h'} 
\bar p_{h,h_0} r_{h_0,h_1} \bar r_{h_1,h'}\\
&=\sum_{h_0;h\leq h_0<h'} \bar p_{h,h_0} \sum_{h_1; h_0\leq h_1\leq h'}
\bar r_{h_0,h_1}r_{h_1,h'}\\
&=\sum_{h_0;h\leq h_0<h'} \bar p_{h,h_0} \delta_{h_0,h'}=0.
\end{align*}

Now we claim that since $\bar{q}+q=0$, there is a unique $q'\in v\Zpv$
such that $q'-\bar{q'}=0$. 

Indeed, we can write $q=\sum_{m,n\in \N} a_{n} v^n$
with $a_{n}\in \Z^\pi$ for $n\in \N$. Then since $q+\bar{q}=0$,
we see that $a_{n}=-\pi^n a_{-n}$ for all $m,n\in \N$.
In particular, $a_{0}=0$ and $a_{n}=0$ if and only
if $a_{-n}=0$ for $n\in \Z$.
Then taking $q'=\sum_{n\in \N} a_{n} v^{n}$,
we see that $q=q'-\bar{q'}$ and $q'\in v\Zpv$. In particular, we can set $p_{h,h'}=q'$,
and then
\[p_{h,h'}=q+\bar{q'}=\sum_{h''; h\leq h'' < h'} \bar p_{h,h''} r_{h'',h'}+\bar{p_{h,h'}}
=\sum_{h''; h\leq h'' \leq h'} \bar p_{h,h''} r_{h'',h'}.\]
\end{proof}

\begin{remark}
For $R$ a commutative ring, we can define a bar-involution on $A=R[x,y]$ by $\bar{f(x,y)}=f(y,x)$.
This involutions descends to a bar-involution on $A_r=R[x,y]/(xy-r)R[x,y]\cong R[x,rx^{-1}]$ for any $r\in R$.
Then the assertions of the lemma apply with $\A$ replaced by $A$ or $A_r$, 
$v$ replaced with $x$, and with $\Zpv$ replaced everywhere by $R[x]$.
In particular, Lemma \ref{lem:semi-linear} (respectively, \cite[\S 24.2]{L93}) is a special case
for $R=\Z^\pi$ and $r=\pi$  (respectively, $R=\Z$ and $r=1$).
\end{remark}

\subsection{The canonical basis of $N(\lambda,\lambda')$}

Let $\lambda,\lambda'\in X^+$. We shall consider the 
following partial order on the set $\cB\times_\pi \cB$:
we say that $(b_1,b_1')\leq (b_2,b_2')$ if 
\[\height |b_1|-\height |b_1'|=\height |b_2|-\height |b_2'|\]
and if we have either
\[\height|b_1|<\height |b_2|\text{ and }\height|b_1'|<\height |b_2'|,\]
or
\[(b_1,b_2)\sim (b_1',b_2').\]
Note that, in particular, $(b_1,b_1')$ and $(b_1,\pi b_1')$ are not comparable
under $\leq$.
For given $\lambda,\lambda'\in X^+$, this induces a partial order
on the set $\cB(\lambda)\times_\pi \cB(\lambda')$.

Then from the definition, we have that
for all $(b_1,b_1')\in B(\lambda)\times_\pi B(\lambda')$,
\[\Psi(b_1^- \eta_{\lambda}\otimes b_1'^+\xi_{-\lambda'})
=\sum_{(b_2,b_2')\in B(\lambda)\times_\pi B(\lambda')}
r_{b_1,b_1';b_2,b_2'} b_2^-\eta_{\lambda}\otimes b_2'^+\xi_{-\lambda'},
\]
where $r_{b_1,b_1';b_2,b_2'}\in \Z[v^{\pm 1}]$
and $r_{b_1, b_1';b_2,b_2'}=0$ unless $(b_1,b_1')\geq (b_2,b_2')$;
in particular, the sum is finite.

Moreover, we note that $r_{b_1,b_1';b_1,b_1'}=1$,
and from $\Psi^2=1$ we see that
\[
\sum_{(b_1,b_1');(b_2,b_2');(b_3,b_3')\in B(\lambda)\times_\pi B(\lambda')}
\bar r_{b_1,b_1';b_2,b_2'} r_{b_2,b_2';b_3,b_3'}
=\delta_{(b_1,b_1');(b_3,b_3')}.
\]

Then $H=B(\lambda)\times_\pi B(\lambda')$ and $r_{b_1,b_1';b_2,b_2'}$
satisfy the assumptions of Lemma \ref{lem:semi-linear}, 
so there exist elements $p_{b_1,b_1';b_2,b_2'}\in \Zpv$ such that
\[p_{b_1,b_1';b_1,b_1'}=1,\]
\[p_{b_1,b_1';b_2,b_2'}=0\text{ unless }(b_2,b_2')\leq (b_1,b_1'),\]
\[p_{b_1,b_1';b_2,b_2'}\in v\Zpv\text{ for } (b_2,b_2')<(b_1,b_1'),\]
\[
p_{b_1,b_1';b_3,b_3'}=
\sum_{(b_2,b_2')\in B(\lambda)\times_\pi B(\lambda')} 
\bar p_{b_1,b_1';b_2,b_2'} r_{b_2,b_2';b_3,b_3'}.
\]

Now recall from Section \ref{subsec:dotU and N} that $\cL$ (resp. $_\A\cL$) is the $\Zpv$-lattice
(resp. $\A$-lattice) with basis $\cB(\lambda,\lambda')=\set{b^+\xi_{-\lambda}\otimes b^-\eta_{\lambda'}:b\times b'\in B(\lambda)\times_\pi B(\lambda')}$.

\begin{theorem}\label{thm:CBforN}
\begin{enumerate}
\item For any $(b,b')\in B(\lambda)\times_\pi B(\lambda')$,
there is a unique element $(b\diamondsuit b')_{\lambda,\lambda'}$ of the $\UU$-module $N(\lambda,\lambda')$
such that 
\[
\Psi((b\diamondsuit b')_{\lambda,\lambda'})=(b\diamondsuit b')_{\lambda,\lambda'}
\text{ and }
(b\diamondsuit b')_{\lambda,\lambda'}-b^-\eta_{-\lambda}\otimes b'^+ \xi_{-\lambda'}\in v\cL.
\]
\item We have $(\pi b\diamondsuit b')_{\lambda,\lambda'}=(b\diamondsuit \pi b')_{\lambda,\lambda'}=
\pi (b\diamondsuit b')_{\lambda,\lambda'}$
\item $(b\diamondsuit b')_{\lambda,\lambda'}$ is equal to
$b^-\eta_{\lambda}\otimes b'^+ \xi_{-\lambda'}$ plus
a $v\Zpv$-linear combination of elements 
$b_1^-\eta_{\lambda}\otimes b_1'^+ \xi_{-\lambda'}$
with $(b_1,b_1')<(b,b')$.
\item The elements $(b\diamondsuit b')_{\lambda,\lambda'}$ form
a $\pi$-basis of $\cL$, ${}_\A\cL$, and $N(\lambda,\lambda')$.
\item The natural homomorphism $\cL\cap \Psi(\cL)\rightarrow \cL/v\cL$
is an isomorphism.
\end{enumerate}
\end{theorem}

\begin{proof}
By the definition of $p_{b,b';b_1,b_1'}$, we see that 
\[(b\diamondsuit b')_{\lambda,\lambda'}=\sum_{(b_1,b_1')\leq (b,b')}
p_{b,b';b_1,b_1'}b_1^-\eta_{\lambda}\otimes b_1'^+ \xi_{-\lambda'}\]
satisfies the requirements of (1) proving existence, 
and the same considerations prove (3). 
(4) is immediate from the fact that the transition matrix
from the $\Qv$-basis $b_1^-\eta_{\lambda}\otimes b_1'^+ \xi_{-\lambda'}$
is unitriangular with entries in $\Z[v]$. (5) follows from (4) and the observation
that $(b\diamondsuit b')_{\lambda,\lambda'}\in \cL\cap \Psi(\cL)$,
and so the map sends $(b\diamondsuit b')_{\lambda,\lambda'}$ to the
basis element $b^-\eta_{\lambda}\otimes b'^+ \xi_{-\lambda'}$
of $\cL/v\cL$, which also implies uniqueness. Finally, uniqueness
implies (2).
\end{proof}

We call the elements $(b\diamondsuit b')_{\lambda,\lambda'}$
the {\em canonical basis} of $N(\lambda,\lambda')$.

\begin{remark}
We may repeat verbatim \S \ref{subsec:R-matrix integrality}
with respect to $\Theta_4$ to obtain a bar-involution
$\Psi_4$ on $N_4(\lambda,\lambda')$.
Then we see that the results of this section
(and in particular Theorem \ref{thm:CBforN})
may be restated with $\Psi$, $N(\lambda,\lambda')$ replaced by $\Psi_4$, $N_4(\lambda,\lambda')$.
When we need to distinguish them, we denote the canonical basis of $N_s(\lambda,\lambda')$ by 
$(b\diamondsuit_s b')_{\lambda,\lambda'}$.
\end{remark}

\subsection{Cancellation and stability}
Our goal in this subsection is to exhibit maps between modules
in $\mathcal F$ which are compatible with canonical bases.
These maps will correspond to a form of cancellation on the pairs 
$(\lambda,\lambda')\in X^+\times X^+$;
namely, the cancellation 
\[(\lambda+\lambda'',\lambda''+\lambda')\mapsto (\lambda,\lambda')\]
can be realized as a $\UU$-module homomorphism
$N(\lambda+\lambda'',\lambda''+\lambda')\rightarrow N(\lambda,\lambda')$.

\begin{proposition}\label{prop:tensor splitting}
Let $\lambda,\lambda'\in X^+$. Write $\eta=\eta_\lambda$,
$\eta'=\eta_{\lambda'}$, and $\eta''=\eta_{\lambda+\lambda'}$.
\begin{enumerate}
\item There is a unique homomorphism of $\UU$-modules 
$\chi:V(\lambda+\lambda')\rightarrow V(\lambda)\otimes V(\lambda')$
such that $\chi(\eta'')=\eta\otimes \eta'$.
\item Let $b\in B(\lambda + \lambda')$. We have
$\chi(b^-\eta'')=\sum_{b_1,b_2} f(b;b_1,b_2) b_1^-\eta\otimes b_2^-\eta'$
where the sum is over $(b_1,b_2)\in B(\lambda)\times_\pi B(\lambda')$
and $f(b;b_1,b_2)\in \Z[v]$.
\item If $b^-\eta'\neq 0$, then $f(b;b,1)=1$ and $f(b;b_1,1)=0$ for
any $b_1\neq b$. If $b^-\eta'=0$, then $f(b; b_1,1))=0$ for any $b_1$.
\end{enumerate}
\end{proposition}

\begin{proof}
The map $\chi$ is the same as the map
$\Phi'(\lambda,\lambda')$ as defined in \cite{CHW2} (with reversed conventions on $\Delta,\Delta';\otimes,\otimes'$),
proving (1). A minor variation on the proof of \cite[Lemma 5.7]{CHW2}
shows that $\Phi'(\lambda,\lambda')$ preserves the crystal lattice.
Combining this with the fact that $\cB\subset \fint$, we see that $f(b;b_1,b_2)\in \Zpv$; moreover, 
up to identifying $\pi (b_1^-\eta\otimes b_2^-\eta')$ with $(\pi b_1^-)\eta\otimes b_2^-\eta'$,
we may assume that $f(b;b_1,b_2)\in \Z[v]$. (3) is immediate from the definition of the
coproduct.
\end{proof}

\begin{remark}\label{remark:alt tensor splitting}
In \cite{CHW2}, there is another map 
\[\Phi(\lambda,\lambda'):V(\lambda+\lambda')\rightarrow V(\lambda)\otimes_4 V(\lambda')\]
and there is a version of the above proposition 
where we replace $\chi,\otimes$ with $\Phi,\otimes_4$ everywhere.
\end{remark}

\begin{proposition}\label{prop:twisted tensor splitting}
Let $\lambda,\lambda'\in X^+$. Write $\xi=\xi_{-\lambda}$,
$\xi'=\xi_{-\lambda'}$, and $\xi''=\xi_{-\lambda-\lambda'}$.
\begin{enumerate}
\item There is a unique homomorphism of $\UU$-modules 
$\chi_4:{}^\omega V(\lambda+\lambda')\rightarrow 
{}^\omega V(\lambda')\otimes {}^\omega V(\lambda)$
such that $\chi_4(\xi'')=\xi'\otimes \xi$.
\item Let $b\in \cB(\lambda + \lambda')_\nu$. We have
\[
\chi_4(b^+\xi'')=\sum_{b_1,b_2} 
\pi^{p(b_1)p(b_2)} f(b;b_1,b_2) b_2^+\xi'\otimes b_1^+\xi
\]
where the sum is over $(b_1,b_2)\in B(\lambda)\times_\pi B(\lambda')$
and $f(b;b_1,b_2)\in \Zpv$.
\item If $b^+\xi'\neq 0$, then $f(b;b,1)=1$ and $f(b;b_1,1)=0$ for
any $b_1\neq b$. If $b^+\xi'=0$, then $f(b;b_1,1))=0$ for any $b_1$.
\end{enumerate}
\end{proposition}

\begin{proof}
By Proposition \ref{prop:tensor splitting} and Remark \ref{remark:alt tensor splitting}, there is a homomorphism
\[\Phi:V(\lambda+\lambda')\rightarrow V(\lambda)\otimes_4 V(\lambda').\]
Then $\Phi$ can also be viewed as a homomorphism ${}^{\omega}V(\lambda+\lambda')\rightarrow {}^{\omega}(V(\lambda)\otimes_4 V(\lambda'))$.
By Lemma \ref{cor:comparing twists}, we have an isomorphism
\[{}^{\omega}(V(\lambda)\otimes_4 V(\lambda'))\rightarrow {}^\omega V(\lambda')\otimes {}^\omega V(\lambda)\qquad
y\otimes z\mapsto \pi^{p(y)p(z)}z\otimes y.
\]
Then taking $\chi'$ to be the composition of these homomorphisms proves (1). The remaining properties follow by the definitions.
\end{proof}

\begin{proposition}\label{prop:contraction}
Let $\eta=\eta_{\lambda}$ and $\xi=\xi_{-\lambda}$.
\begin{enumerate}
\item There is a unique homomorphism of $\UU$-modules
$\delta_\lambda:N(\lambda,\lambda)\rightarrow \Qvp$,
where $\Q(v)$ is a $\UU$-module under the counit map,
such that $\delta_\lambda(\xi\otimes \eta)=1$.

\item Let $b,b'\in B(\lambda)$. Then 
$\delta_\lambda(b^+\xi\otimes b'^-\eta)=1$
if $b=b'=1$ and is in $v\Zpv$ otherwise.
\end{enumerate}
\end{proposition}

\begin{proof}
For such a map to exist, we would need
\[\Delta(E_i) (x\otimes y)
=E_i x\otimes \tK_{i} y+ \pi_i^{p(x)} x\otimes F_i y\in \ker(\delta_\lambda)\text{ for all }x\otimes y\in N(\lambda,\lambda);
\]
\[\Delta(F_i) (x\otimes y)
=F_i x\otimes y+ \pi_i^{p(x)}\tJ_i\tK_{i} x\otimes \pi_i \tJ_i E_i y\in \ker(\delta_\lambda)\text{ for all }x\otimes y\in N(\lambda,\lambda);
\]
\[\Delta(K_\mu-1) (x\otimes y)
=K_{\mu}x\otimes K_{-\mu} y-x\otimes y\in \ker(\delta_\lambda)\text{ for all }x\otimes y\in N(\lambda,\lambda);
\]
\[\Delta(J_\mu-1) (x\otimes y)
=J_{\mu}x\otimes J_{\mu} y-x\otimes y\in \ker(\delta_\lambda)\text{ for all }x\otimes y\in N(\lambda,\lambda);
\]
The following statement is equivalent to (1).
There is a unique bilinear pairing 
$[,]: V(\lambda)\times V(\lambda)\rightarrow \Qvp$
satisfying $[\eta,\eta]=1$, 
\[[E_ix,\tK_{i} y]=-\pi_i^{p(x)}[x, F_i y],\quad
[F_i x,y]=-\pi_i^{p(x)}[\tJ_i\tK_{i}x,\pi_i\tJ_iE_i y]\] 
\[[K_{\mu} x,y]=[x,K_\mu y],\quad 
[J_\mu x,y]=[x,J_\mu y].\]
We may rewrite the conditions as
\[[ E_ix, y]=-\pi_i^{p(x)}[ x, F_i \tK_{-i} y],\quad
[F_i x,y]=-\pi_i^{p(x)}[x,\pi_i\tK_{i}F_i y]\] 
\[[K_{\mu} x,y]=[x,K_\mu y],\quad 
[J_\mu x,y]=[x,J_\mu y].\]

Let $\clubsuit:\UU\rightarrow \UU$ be the $\Qvp$-linear map
defined by
\[\clubsuit(E_i)=-F_i\tK_{-i},\quad 
\clubsuit(F_i)=-\pi_i \tK_i E_i,
\quad \clubsuit(K_\mu)=K_\mu,\quad \clubsuit(J_\mu)=J_\mu,\]
\[\clubsuit(xy)=\pi^{p(x)p(y)}\clubsuit(y)\clubsuit(x).\]
To see this is a well-defined map, we note that
$\clubsuit=\omega S$ where $S$
is the antipode of $\Delta$; that is, the $\Qvp$-linear map 
satisfying
\[S(E_i)=-E_i\tK_i,\quad S(F_i)=-\tJ_i\tK_{-i}F_i ,
\quad S(K_\mu)=K_{-\mu}, \quad S(J_\mu)=J_\mu,\]
\[S(xy)=\pi^{p(x)p(y)}S(y)S(x).\]
(Note that this is not the antipode defined in \cite[\S 2.4]{CHW1}, as that
map corresponds to the coproduct $\Delta_1$.)

Then we see that (1) is equivalent to proving the
existence of a unique bilinear pairing
$[,]: V(\lambda)\times V(\lambda)\rightarrow \Qvp$
satisfying $[\eta,\eta]=1$ and $[ux,y]=\pi^{p(x)p(u)}[x,\clubsuit(u)y]$.
This is proven using a standard argument: in brief, the restricted dual $V(\lambda)^*$
has a $\UU$-action $(uf)(x)=\pi^{p(u)p(f)}f(\clubsuit(x)u)$
under which we have an isomorphism $V(\lambda)^*\cong V(\lambda)$,
and $[\cdot,\cdot]$ is the natural pairing.

Let $(\cdot,\cdot)$ be the polarization on $V(\lambda)$.
We show by induction on $\height \nu\geq 0$ that
\begin{equation}\label{eq:[]vs()}
[x,y]=(-1)^{\height \nu} \pi^{\mathbf p(\nu)}
\pi_{\nu} v_{\nu} (x, y)
\end{equation}
for $x,y\in V(\lambda)_{\lambda-\nu}$, where here we set 
$\mathbf p(\sum_t i_t)=\sum_{s<t}p(i_s)p(i_t)$.

This is obvious for $\nu=0$. Assume that $\height \nu\geq 1$.
Then we can assume that $x=F_ix'$ for some $i$ such that $\nu_i>0$.
Then by induction, we compute
\begin{align*}
[x,y]&=[F_ix',y]=-\pi_i^{p(x')}[x',\pi_i\tK_iE_i y]
=-\pi_i^{p(\nu-i)}\pi_i[x',\tK_iE_i y]\\
&=-\pi_i^{p(\nu-i)}\pi_i (-1)^{\height (\nu-i)} \pi_i^{\mathbf p(\nu-i)}
\pi_{\nu-i}v_{\nu-i} (x',\tK_i E_i y)\\
&=(-1)^{\height \nu} \pi^{\mathbf p(\nu)}\pi_{\nu} v_{\nu}(F_ix',y)
=(-1)^{\height \nu} \pi^{\mathbf p(\nu)}\pi_{\nu} v_{\nu}(x,y).
\end{align*}
This proves that \eqref{eq:[]vs()} holds. Now recall 
from Lemma \ref{lem:extra CB props} (2) that $(b\eta,b'\eta)\in \Zpv$
for any $b,b'\in \cB$. Combining this with \eqref{eq:[]vs()},
we see that (2) follows.
\end{proof}

We have demonstrated the existence of the maps $\chi,\chi_4$
Let $\lambda,\lambda',\lambda''\in X^+$. We define
a $\UU$-module homomorphism
\[t:N(\lambda+\lambda',\lambda'+\lambda'')
\rightarrow N(\lambda,\lambda'')\]
defined as the composition 
\[t=(1\otimes \delta_\lambda\otimes 1)\circ (\chi\otimes \chi').\]

\begin{lemma}
We have $\Psi t=t\Psi$.
\end{lemma}
\begin{proof}
We write $\eta=\eta_{\lambda+\lambda'}$ and
$\xi=\xi_{-\lambda'-\lambda''}$.
Since any element of $N(\lambda+\lambda',\lambda'\otimes \lambda'')$
is of the form $u(\eta\otimes\xi)$, it is enough to check that
\[
t \Theta( \bar{ u( \eta\otimes \xi )} )=
\Theta( \bar{ t( u( \eta\otimes\xi ))}).
\]
Well, on one hand
\[
t \Theta( \bar{ u( \eta\otimes \xi )} )
=t(u \Theta(\eta\otimes \xi))
=ut((\eta\otimes \xi))
=u(\eta_{\lambda}\otimes \xi_{-\lambda''}).
\]
On the other hand,
\[
\Theta( \bar{ t(u( \eta\otimes \xi ))} )
=\Theta(\bar{ut(\eta\otimes\xi)})
=u\Theta(\eta_\lambda\otimes \xi_{-\lambda''})
=u(\eta_\lambda\otimes \xi_{-\lambda''}).
\]
The lemma is proved.
\end{proof}

\begin{lemma}
\begin{enumerate}
\item Let $b\in \cB(\lambda)$ and $b''\in \cB(\lambda'')$.
Then
\[
t(b^-\eta_{\lambda+\lambda'}\otimes b''^+\xi_{-\lambda'-\lambda''})
= b^-\eta_{\lambda}\otimes b''^+\xi_{-\lambda''} \text{ mod }
v\cL(\lambda,\lambda'').
\]
In particular,
$t((b\diamondsuit b')_{\lambda+\lambda',\lambda'+\lambda''})
=(b\diamondsuit b')_{\lambda,\lambda''}$ mod $v\cL(\lambda,\lambda'')$.
\item Let $b\in \cB(\lambda+\lambda')$ and 
$b''\in \cB(\lambda'+ \lambda'')$. Assume that either 
$b\notin \cB(\lambda)$ or $b''\notin \cB(\lambda'')$.
Then
\[
t(b^-\eta_{\lambda+\lambda'}\otimes b''^+\xi_{-\lambda'-\lambda''})
= 0 \text{ mod }
v\cL(\lambda,\lambda'').
\]
In particular,
$t((b\diamondsuit b')_{\lambda+\lambda',\lambda'+\lambda''})
=0$ mod $v\cL(\lambda,\lambda'')$.
\item $t$ is surjective.
\end{enumerate}
\end{lemma}

\begin{proof}
Parts (1) and (2) follow from Propositions 
\ref{prop:tensor splitting}-\ref{prop:contraction} and the
definition of $t$. In particular,
note that $t(\eta_{\lambda+\lambda'}\otimes \xi_{-\lambda'-\lambda''})
=\eta_{\lambda}\otimes \xi_{-\lambda''}$,
which generates $N(\lambda,\lambda'')$.
\end{proof}

\begin{proposition}
\begin{enumerate}
\item Let $b\in \cB(\lambda)$ and $b''\in \cB(\lambda'')$.
Then \[t((b\diamondsuit b')_{\lambda+\lambda',\lambda'+\lambda''})
=(b\diamondsuit b')_{\lambda,\lambda''}.\]
\item Let $b\in \cB(\lambda+\lambda')$ and $b''\in \cB(\lambda'+\lambda'')$.
Assume that $b\notin \cB(\lambda)$ or $b''\notin \cB(\lambda'')$.
Then \[t((b\diamondsuit b')_{\lambda+\lambda',\lambda'+\lambda''})
=0.\]
\end{enumerate}
\end{proposition}

\begin{proof}
The differences of the two sides of the claimed equalities 
in (1) and (2) lie in $v\cL(\lambda,\lambda'')$ and are 
fixed by $\Psi$, hence the difference is zero.
\end{proof}
\subsection{The canonical basis of $\dotU$}
We are now in a position to produce a canonical basis for $\dotU$ 
which descends to the canonical bases for modules in $\mathcal F$.
\begin{theorem}\label{thm:cbdotU}
Let $\zeta\in X$
and $(b,b'')\in \cB\times_\pi \cB$.
\begin{enumerate}
\item There is a unique element $u=b\diamondsuit_\zeta b''\in \dotUint$
such that
\[
\Delta(u)(\eta_{\lambda}\otimes \xi_{-\lambda''})=
(b\diamondsuit b'')_{\lambda,\lambda''}
\]
for any $\lambda,\lambda''\in X^+$ such that
$b\in B(\lambda)$, $b''\in B(\lambda'')$,
and $\zeta=\lambda-\lambda''$.
\item If $\lambda,\lambda''\in X^+$ are such
that $\lambda-\lambda''=\zeta$ and either
$b\notin \cB(\lambda)$ or $b''\notin \cB(\lambda'')$,
then 
\[
\Delta(b\diamondsuit_\zeta b'')(\eta_{\lambda}\otimes \xi_{-\lambda''})=0.
\]
\item $\bar{b\diamondsuit_\zeta b''}=b\diamondsuit_\zeta b''$.
\item The elements $b\diamondsuit_\zeta b''$, for various
$\zeta,b,b''$ as above, form a $\Qvp$-basis of $\dotU$
and a $\A$-basis of $\dotUint$.
\end{enumerate}
\end{theorem}

\begin{proof}
First recall that, throughout the paper, we assume the root datum is $Y$-regular.
Then we can find $\lambda,\lambda''\in X^+$
such that $b\in B(\lambda)$, $b''\in B(\lambda'')$,
and $\lambda-\lambda''=\zeta$.

For any integers $N_1$, $N_2$, let $P(N_1,N_2)$
be the $\A$-submodule of $\dotUint$
spanned by the elements $b_1^-b_2^+1_{\zeta}$
where $b_1$ and $b_2$ run through the elements of $\cB$
such that $\height |b_1|\leq N_1$, $\height |b_2|\leq N_2$,
and $|b_1|-|b_2|=|b|-|b''|$.

Recall that any element of $N(\lambda,\lambda'')$
of the form $\beta^-\eta_\lambda\otimes \beta'^+\xi_{-\lambda}$
with $\beta,\beta'\in \cB$ is equal to 
$u_1(\eta_\lambda\otimes\xi_{-\lambda})$ for some
$u_1\in P(\height |\beta|, \height |\beta'|)$;
moreover, $u_1$ can be taken to be equal to
$\beta^-\beta'^+1_\zeta$ plus an element in
$P(\height |\beta|-1, \height |\beta'|-1)$.
In particular, we see that
$(b\diamondsuit b'')_{\lambda,\lambda''}$
is of the form $u(\eta_{\lambda}\otimes \xi_{-\lambda''})$
for some $u\in P(\height |b|,\height |b''|)$; moreover,
$u$ can be taken to be equal to $b^-b''^+1_\zeta$ plus
an element in $P(\height |b|-1,\height |b''|-1)$.

Assume that $u$ is such an element and $u'$ is another
such element. Then $(u-u')(\eta_\lambda\otimes \xi_{-\lambda''})=0$,
and so by Proposition \ref{prop:dotU onto N}
\[(u-u')\in \sum_{i,n>\ang{i,\lambda}} \dotUint F_i^{(n)} 1_\zeta+
\sum_{i,n>\ang{i,\lambda''}} \dotUint E_i^{(n)} 1_\zeta. \]
However, since $u-u'\in P(\height |b|,\height |b''|)$,
we must have $u-u'=0$ if $\ang{i,\lambda}>\height |b|$ and
$\ang{i,\lambda''}>\height|b''|$ for all $i\in I$.
For such $\lambda,\lambda''$ the element $u$ is uniquely
determined, and we denote it by $u_{\lambda,\lambda''}$.

Assume now that $\lambda,\lambda''\in X^+$ satisfy $b\in \cB(\lambda)$,
$b''\in \cB(\lambda'')$, and $\lambda-\lambda''=\zeta$.
Let $\lambda'\in X^+$ such that $\lambda>>0$ so that
$u'=u_{\lambda+\lambda',\lambda'+\lambda''}$ is defined.
Then 
\[
u'(\eta_{\lambda}\otimes \xi_{-\lambda''})
=u't(\eta_{\lambda+\lambda'}\otimes \xi_{-\lambda'-\lambda''})
=t((b\diamondsuit b'')_{\lambda+\lambda',\lambda'+\lambda''})
=(b\diamondsuit b'')_{\lambda,\lambda''}.
\]
Then $u_{\lambda,\lambda''}$ is independent of $\lambda,\lambda''$
if it is defined, so we may
denote it as $u$ without specifying $\lambda$, $\lambda''$.
In particular, this element satisfies the requirements of (1),
proving existence and uniqueness.

This argument also proves (2), since in this case 
we may pick $\lambda'$ so that $b\in \cB(\lambda +\lambda')$ and $b\in \cB(\lambda'+\lambda'')$,
and then $t((b\diamondsuit b'')_{\lambda+\lambda',\lambda'+\lambda''})=0$.

The bar-invariance of the canonical basis of $N(\lambda,\lambda'')$
and the uniqueness of the element $u$ shows that $\bar{u}=u$,
and hence (3) holds.

Finally, the uniqueness of $u$ forces
$b\diamondsuit_\zeta b''=b^-b''^+1_\zeta$ modulo $P(\height |b|-1,\height |b''|-1)$
Since $b^-b''^+ 1_\zeta$ forms a basis of $\dotU$, 
the transition matrix from $b^-b''^+1_\zeta$
to $b\diamondsuit_\zeta b''$ may be made upper unitriangular (by a suitable ordering), 
hence $b\diamondsuit_\zeta b''$ forms a basis.
\end{proof}

We let $\dot\cB=\set{(b\diamondsuit_\zeta b':(b,b')\in \cB\times_\pi \cB, \zeta\in X}$
and call this {\em the canonical basis of $\dotU$}.

\begin{corollary}\label{cor:dotCB Delta4 action}
Let $\zeta\in X$
and $(b,b')\in \cB\times_\pi \cB$.
For any $\lambda,\lambda'\in X^+$ such that
$b\in B(\lambda)$, $b'\in B(\lambda')_\nu$,
and $\zeta=\lambda-\lambda'$,
\[
\Delta_4(b\diamondsuit_\zeta b')(\eta_{\lambda}\otimes \xi_{-\lambda'})=
\pi^{\ang{\tilde\nu,\lambda}}(b\diamondsuit_4 b')_{\lambda,\lambda'}.
\]
\end{corollary}

\begin{proof}
Let $u=b\diamondsuit_\zeta b'$.
Then by definition, $\Delta(u)\eta_{\lambda}\otimes\xi_{-\lambda''}=(b\diamond b'^+)_{\lambda,\lambda'}$,
and $(b\diamondsuit b')_{\lambda,\lambda'}-b^+\eta_{\lambda}\otimes b^- \xi_{-\lambda'}\in v\cL$.
Therefore, applying the isomorphism in Lemma \ref{lem:N34iso} (which obviously preserves $\cL$),
\[\Delta_4(u)\eta_{\lambda}\otimes\xi_{-\lambda''}-\pi^{\ang{\tilde{\nu},\lambda}}b^+\eta_{\lambda}\otimes b^- \xi_{-\lambda'}\in v\cL.\]
On the other hand, $\Psi_4(\Delta_4(u)\eta_{\lambda}\otimes\xi_{-\lambda''})=\Delta_4(u)\eta_{\lambda}\otimes\xi_{-\lambda''}$
by the bar-invariance of $u$. Then by uniqueness, we have $\Delta_4(u)=\pi^{\ang{\tilde\nu,\lambda}}(b\diamondsuit_4 b')_{\lambda,\lambda'}$.
\end{proof}

\begin{example}
Suppose that $I=I_{\one}=\set{i}$. Then $\cB=\set{\pi^\epsilon \theta_i^{(a)}:a\in \N,\epsilon\in\set{0,1}}$.
Let $a,b\in \N$ and suppose that $n\geq a+b$. Then using similar computations
as in the proof of \cite[Theorem 6.2]{CW},
\begin{align*}
(\theta_i^{(a)}\diamondsuit_{n-2a} \theta_i^{(b)}&)=F_i^{(a)} E_i^{(b)}1_{n}\\
(\theta_i^{(a)}\diamondsuit_{2b-n} \theta_i^{(b)})&=\pi^{ab}E_i^{(b)} F_i^{(a)}1_{-n}.
\end{align*}
We note that the though this $\pi$-basis matches the one in {\em loc. cit.}, 
the $(\cB\times_\pi \cB)\times X$ labeling differs because
several conventions differ.
In particular, this canonical basis agrees with the categorical canonical basis produced in \cite{EL}.
\end{example}

\section{Inner product on $\dotU$}

In this section, we will construct a bilinear form on $\dotU$.
This bilinear form is $\omega$-invariant and $\rho$-invariant, and can be viewed
as a limit of bilinear forms on the vector spaces $N(\lambda,\lambda)'$.
In particular, the canonical basis is shown to be $\pi$-almost-orthonormal
with respect to this form.

\subsection{Review of automorphisms}

Let us recall some additional automorphisms of $\UU$ (and $\dotU$) which will play a role in the following computations.
We already defined $\omega$ and $\rho$ in \S\ref{subsec:cqg}.

Let $\tau_1:\UU\rightarrow \UU$ be the anti-involution defined by
\[\tau_1(E_i)=v_i^{-1}\tK_{-i} F_i,\quad \tau_1(F_i)=v_i^{-1} \tK_i E_i,\quad \tau_1(K_\mu)=K_{\mu},\quad \tau_1(J_\mu)=J_\mu.\]
We set $\bar\tau_1$ to be the map $u\mapsto \bar{\tau_1(\bar u})$; that is, the map satisfying
\[\bar\tau_1(E_i)=\pi_i v_i^{-1}F_i\tJ_i\tK_i,\quad \tau_1(F_i)=\pi_iv_i^{-1} E_i\tJ_i\tK_{-i},\quad \tau_1(K_\mu)=K_{\mu},\quad \tau_1(J_\mu)=J_\mu.\]

\begin{lemma}\label{lem:rho tau omega ids}
We have the identities
\[\rho\tau_1=\bar\tau_1\rho,\quad \omega^{-1}\tau_1=\tau_1\omega.\]
\end{lemma}
\begin{proof}
It suffices to check these identities on the generators, and therefore, since all compositions considered fix $J_\mu$ and $K_\mu$,
we only need to check on $E_i$ and $F_i$ for $i\in I$.

Well,for $\rho\tau_1=\bar\tau_1\rho$:
\[\rho(\tau_1(E_i))=\rho(v_i^{-1}\tK_{-i} F_i)=v_i^{-1}F_i\tK_i,\quad \bar\tau_1(\rho(E_i)=\bar\tau_1(\pi_i\tJ_iE_i)=v_i^{-1}F_i\tK_i,\]
\[\rho(\tau_1(F_i))=\rho(v_i^{-1}\tK_{i} E_i)=\pi_i v_i^{-1}E_i\tJ_i\tK_{-i},\quad \bar\tau_1(\rho(F_i)=\bar\tau_1(E_i)=v_i^{-1}E_i\tJ_i\tK_{-i}.\]
Checking the left-hand equations to the right-hand equations verifies the identity.

Similarly, for $\omega^{-1}\tau_1=\tau_1\omega$:
\[
\omega^{-1}(\tau_1(E_i))=\omega^{-1}(v_i^{-1}\tK_{-i} F_i)=v_i^{-1}\tK_iE_i,\quad 
\tau_1(\omega(E_i))=\tau_1(F_i)=v_i^{-1}\tK_{i} E_i,
\]
\[
\omega^{-1}(\tau_1(F_i))=\omega(v_i^{-1}\tK_{i} E_i)=\pi_i v_i^{-1}\tJ_i\tK_{-i}F_i,\quad 
\tau_1(\omega(F_i)=\tau_1(\pi_i \tJ_i E_i)=\pi_i v_i^{-1}\tJ_i\tK_{-i}F_i.
\]
Comparing the left-hand equations to the right-hand equations verifies the identity.
\end{proof}

\subsection{The bilinear form}

Let $(\cdot,\cdot):\ff\times \ff\rightarrow \Qvp$ denote the bilinear form $\set{\cdot,\cdot}$ in \cite{CHW1}.
Let $\tau_1:\UU\rightarrow \UU$ be as defined in \cite{CHW2}, and note it induces a
homomorphism $\tau_1:\dotU\rightarrow \dotU$. Let $\bar \tau_1$ be the bar-conjugate
homomorphism.

\begin{theorem}\label{thm:dotbilinearform}
There exists a unique $\Qvp$-bilinear pairing $(\cdot,\cdot):\dotU\times \dotU\rightarrow \Qvp$
such that the following hold.
\begin{enumerate}
\item $(1_{\lambda_1} x 1_{\lambda_2},  1_{\lambda_1'}x' 1_{\lambda_2'})$ is zero unless $\lambda_1=\lambda_1'$, $\lambda_2=\lambda_2'$;
\item $(ux,y)=(x,\tau_1(u)y)$ for all $x,y\in \dotU$ and $u\in \UU$;
\item $(x^-1_\lambda, x'^- 1_\lambda)=(x,x')$ for all $x,x'\in \ff$ and all $\lambda$;
\item We have $(x,y)=(y,x)$.
\end{enumerate}
\end{theorem}

\begin{proof}
The proof is essentially the same as in \cite[Theorem 26.1.2]{L93}.
\end{proof}

\begin{proposition}
We have $(xu,y)=(x,y\bar \tau_1(u))$ for all $x,y\in \dotU,u\in \UU$.
\end{proposition}
\begin{proof}
It suffices to prove this for the generators; it is clear for $K_\mu$, $J_\mu$.
It remains to verify that
\[(xE_i,y)=(x,\pi_i v_i^{-1} y F_i\tJ_i\tK_i),\qquad (xF_i,y)=(x,\pi_i v_i^{-1} yE_i \tJ_i\tK_{-i}).\]
We may further assume that $x=u' 1_\zeta$ where $u'\in U$ and $\zeta\in X$.
Then setting $y'=\tau_1(u')y$, we see that the previous equalities follow from
\[(1_\zeta E_i, y')=(1_\zeta, \pi_i v_i^{-1} y' F_i\tJ_i\tK_i),\quad (xF_i,y)=(x,\pi_i v_i^{-1} yE_i\tJ_i\tK_{-i}).\]
Once more, we can assume that $y'=\tau_1(y_1^-)y_2^-1_\zeta'$
for homogeneous $y_1,y_2\in \ff$, so it suffices to show that
\[
(y_1^-E_i 1_{\zeta-i'}, y_2^-1_\zeta')=\pi_i^{1+\ang{i,\zeta'+i'}}v_i^{-1+\ang{i,\zeta'+i'}} (y_1^-1_\zeta, y_2^-F_i1_{\zeta'+i'}),\tag{a}
\]
\[
(y_1^-F_i 1_{\zeta+i'}, y_2^-1_\zeta')=\pi_i^{1+\ang{i,\zeta'-i'}}v_i^{-1-\ang{i,\zeta'-i'}} (y_1^-1_\zeta, y_2^-E_i1_{\zeta'-i'}).\tag{b}
\]
By symmetry, (a) and (b) are equivalent, so we shall prove (a). Then we may assume $\zeta'=\zeta-i'$
and $|y_1|=|y_2|+i$.

Recall the $\Qvp$-linear differentials ${}_ir,r_i:\ff\rightarrow \ff$ defined by ${}_i r(\theta_j)=r_i(\theta_j)=\delta_{ij}$,
\[{}_i r(xy)={}_ir(x)y+\pi^{p(x)p(i)}v^{-i\cdot |x|}x\ {}_ir(y),\qquad r_i(xy)=\pi^{p(x)p(i)}v^{-i\cdot |x|}r_i(x)y+xr_i(y).\]
Using a variant on \cite[Proposition 2.2.2]{CHW1} (essentially with $q$ replaced by $v^{-1}$) we have
\[
E_i y_1^- 1_{\zeta'}-\pi_i^{p(y_1)} y_1^- E_i 1_{\zeta'}
=\frac{\pi_i^{p(y_1)-p(i)}r_i(y_1)^-\tJ_i \tK_i   - \tK_{-i}\ {}_ir(y_1)^-}{\pi_i v_i-v_i^{-1}} 1_\zeta,
\]
from whence we see that
\[
y_1^- E_i 1_{\zeta'}
=\pi_i^{p(y_1)}E_i y_1^- 1_{\zeta'} +
\frac{\pi_ir_i(y_1)^-\tJ_i \tK_i   - \pi_i^{p(y_1)}\tK_{-i}\ {}_ir(y_1)^-}{v_i^{-1}-\pi_i v_i} 1_\zeta.
\]
Now note that
\[
(\pi_i^{p(y_1)}E_iy^-1_{\zeta'},y_2^-1_{\zeta'})=\pi_i^{p(y_1)}v_i^{1+\ang{i,|y_2|-\zeta'}} (y_1^-1_{\zeta'}, F_iy_2^-1_{\zeta'})
=\pi_i^{p(y_1)}v_i^{1+\ang{i,|y_2|-\zeta'}}(y_1,\theta_iy_2)
\]
and that
\begin{align*}
&\parens{\frac{\pi_i r_i(y_1)^-\tJ_i \tK_i   - \pi_i^{p(y_1)}\tK_{-i}\ {}_ir(y_1)^-}{v_i^{-1}-\pi_i v_i} 1_\zeta,y_2^-1_{\zeta'} }\\
&=(\theta_i,\theta_i) ( \pi_i^{1+\ang{i,\zeta'}} v_i^{ 1 + \ang{ i,\zeta' } } r_i(y_1)^- 1_{\zeta'}
-\pi_i^{p(y_1)} v_i^{ 1+ \ang{ i,|y_1|-i-\zeta' } } {}_i r(y_1)^- 1_{\zeta'},y_2^-1_{\zeta'})\\
&=\pi_i^{1+\ang{i,\zeta'}} v_i^{ 1 + \ang{ i,\zeta' } } (\theta_i,\theta_i)(r_i(y_1),y_2)
-\pi_i^{p(y_1)} v_i^{ 1+ \ang{ i,|y_1|-i-\zeta' } } (\theta_i,\theta_i)({}_i r(y_1),y_2)\\
&=\pi_i^{1+\ang{i,\zeta'}} v_i^{ 1 + \ang{ i,\zeta' } } (y_1,y_2\theta_i)
-\pi_i^{p(y_1)} v_i^{ 1+ \ang{ i,|y_1|-i-\zeta' } } (y_1,\theta_i y_2).
\end{align*}
Then we see that
\[(y_1^- E_i 1_{\zeta'},y_2^-1_\zeta)=\pi_i^{1+\ang{i,\zeta'}} v_i^{ 1 + \ang{ i,\zeta' } } (y_1,y_2\theta_i)
=\pi^{1+\ang{i,\zeta'+i'}}  v_i^{ -1 + \ang{ i,\zeta' +i'} } (y_1^-1_{\zeta'+i},y_2^-F_i 1_{\zeta'+i'}),\]
which proves (a).
\end{proof}

\begin{proposition}
We have $(\rho(x),\rho(y))=(x,y)$ for all $x,y\in \dotU$.
\end{proposition}

\begin{proof}
It suffices to show that $x,y\mapsto (\rho(x),\rho'(y))$ satisfies the defining properties of $(\cdot,\cdot)$.
All of these are obvious except Theorem \ref{thm:dotbilinearform} (2).
However, this follows from the previous proposition and the fact that $\rho\tau_1=\bar{\tau}_1\rho$.
\end{proof}

\begin{lemma}\label{lem:dotbilform on Uplus}
For $x, x'\in \ff_\nu$, $\nu\in N[I]$ and $\lambda\in X$, $(x^+1_\lambda,x'^+1_\lambda)=\pi_\nu\pi^{\ang{\tilde\nu,\lambda}}(x,x')$.
\end{lemma}
\begin{proof}
We have
\[(x^+1_\lambda, x'^+1_\lambda)=(1_\lambda,\tau_1(x^+)x'^+1_\lambda)
=(1_{-\lambda},1_{-\lambda}\rho(x'^+)\rho\tau_1(x^+))
=(1_{-\lambda},\rho(x'^+)\bar\tau_1\rho(x^+)1_{-\lambda}).\]
Then we may rearrange to obtain
\[(x^+1_\lambda, x'^+1_\lambda)=(\tau_1\rho(x'^+)1_{-\lambda},\bar\tau_1\rho(x^+)1_{-\lambda}).\]

Now note that for any $x\in \ff_\nu$, 
$\tau_1\rho(x^+)=\pi_\nu v_{\nu} v^{\bullet(\nu)} \rho(x)^-\tJ_\nu\tK_{-\nu}$
where $\bullet(i_1+\ldots+i_n)=\sum_{s<t} i_s\cdot i_t$. 
Likewise, noting that $\bullet(\nu)\in 2\Z$, we obtain
$\bar\tau_1\rho(x^+)=v_{-\nu} v^{-\bullet(\nu)} \rho(x)^-\tK_{\nu}$.
Then
\[(\tau_1\rho(x'^+)1_{-\lambda},\bar\tau_1\rho(x^+)1_{-\lambda})
=\pi_\nu\pi^{\ang{\tilde\nu,\lambda}}(x,x')
\]
\end{proof}

\begin{proposition}
We have $(\omega(x),\omega^{-1}(y))=(x,y)$ for all $x,y\in \dotU$.
\end{proposition}

\begin{proof}
It suffices to show that $x,y\mapsto (\omega(x),\omega^{-1}(y))$ satisfies the defining properties of $(\cdot,\cdot)$.
It is clear that (1) holds. 
Since $\omega\tau_1=\tau_1\omega$, it is clear that (2) is satisfied.
Note that $\omega(x^-1_\lambda)=\pi_{\nu} \pi^{\ang{\tilde{\nu},\lambda}} x^+1_\lambda$,
while $\omega^{-1}(x'^-1_\lambda)=x'^+1_\lambda$ for each $x,x'\in \ff_\nu$, whence
(3) holds by the previous lemma. 
Finally, we note that $\omega^2(y)=\pi_\nu\pi^{\ang{\tilde{\nu},\lambda}} x$
for any $x\in \dotU 1_\lambda$, and hence $(\omega^2(x),\omega^2(y))=(x,y)$ for any $x,y\in \dotU$,
proving (4).
\end{proof}

\begin{example}
We compute the following inner products. Let $\lambda_i=\ang{i,\lambda}$.
\begin{align*}
&(F_i^{(k)}1_\lambda,F_i^{(k)}1_\lambda)=\pi^{\binom{k}{2}}\prod_{s=1}^k \frac{1}{1-(\pi_i v_i^2)^s}\\
&(E_i^{(k)}1_\lambda,E_i^{(k)}1_\lambda)=\pi_i^{\binom{k+1}{2}+k\lambda_i}\prod_{s=1}^k \frac{1}{1-(\pi_i v_i^2)^s}\\
&(E_iF_i1_\lambda, 1_\lambda)=\frac{v_i^{1-\lambda_i}}{1-\pi_i v_i^2}\\
&(E_iF_i1_\lambda, E_iF_i1_\lambda)=\pi^{\lambda_i-1}\frac{1+(\pi_i v_i^2)^{1-\lambda_i}}{(1-\pi_iv_i^2)^2}\\
&(E_iF_i1_\lambda, F_iE_i1_\lambda)=\pi^{\lambda_i}\frac{1+\pi_i v_i^2}{(1-\pi_iv_i^2)^2}
\end{align*}

We note that under the identification $q^2=\pi v^2$, these values
are formally similar to the values of the analogous bilinear form on $\dotU|_{\pi=1}$
over $\Q(q)$ (but with an additional factor of some power of $\pi$).

\if 0\begin{align*}
(E_iF_i1_\lambda, E_iF_i1_\lambda)
&=v^{1-\lambda_i}(F_i1_\lambda, F_iE_iF_i1_\lambda)\\
&=\pi v^{1-\lambda_i}(F_i1_\lambda,E_iF_i^2-[\lambda_i-2]_iF_i1_\lambda)\\
&=\pi v^{4-2\lambda_i}(F_i^21_\lambda, F_i^21_\lambda)-\pi v^{1-\lambda_i}[\lambda_i-2]_i(F_i1_\lambda, F_i1_\lambda)\\
&=v_i^{4-2\lambda_i}\frac{1}{1-\pi_iv_i^2}\frac{1}{1-v_i^4}[2]_i^2-\pi v^{1-\lambda_i}[\lambda_i-2]_i\frac{1}{1-\pi_iv_i^2}\\
&=\frac{1}{1-\pi_iv_i^2}\frac{1}{1-v_i^4}\parens{
v_i^{4-2\lambda_i} [2]_i^2 - \pi_i v_i^{1-\lambda_i}(1-v_i^4)\frac{(\pi_iv_i)^{\lambda_i-2}-v_i^{2-\lambda_i}}{\pi_iv_i-v_i^{-1}} 
}\\
&=\frac{1}{1-\pi_iv_i^2}\frac{1}{1-v_i^4}[2]\parens{
v_i^{4-2\lambda_i} [2]_i+ \pi_i v_i^{3-\lambda_i}((\pi_iv_i)^{\lambda_i-2}-v_i^{2-\lambda_i}) 
}\\
&=\frac{1}{1-\pi_iv_i^2}\frac{1}{1-v_i^4}[2]\parens{
\pi_i v_i^{5-2\lambda_i}+v_i^{3-2\lambda_i}+ \pi_i^{\lambda_i-1} v_i-\pi_i v_i^{5-2\lambda_i} 
}\\
&=\frac{1}{1-\pi_iv_i^2}\frac{1}{1-v_i^4}[2]\parens{
v_i^{3-2\lambda_i}+ \pi_i^{\lambda_i-1} v_i
}\\
&=\pi^{\lambda_i-1}\frac{1}{1-\pi_iv_i^2}\frac{1}{1-v_i^4}(1+\pi v_i^2)(1+(\pi_i v_i^2)^{1-\lambda_i})\\
&=\pi^{\lambda_i-1}\frac{1+(\pi_i v_i^2)^{1-\lambda_i}}{(1-\pi_iv_i^2)^2}
\end{align*}
\fi 
\end{example}

\begin{proposition}\label{prop:altdotbilform}
Let $\tau_1'$ be the $\Qvp$-linear anti-automorphism of $\UU$ defined by
\[\tau'_1(E_i)= v_i \tK_i F_i,\quad \tau'_1(F_i)=v_i\tK_i^{-1} E_i,\quad \tau'_1(K_\nu)=K_\nu,\quad \tau'_1(J_\nu)=J_\nu.\]
Then there is a $\Qvp$-bilinear pairing $(\cdot,\cdot)':\dotU\times\dotU\rightarrow \Qvp$ 
such that the following hold.
\begin{enumerate}
\item $(1_{\lambda_1} x 1_{\lambda_2},  1_{\lambda_1'}x' 1_{\lambda_2'})'$ is zero unless $\lambda_1=\lambda_1'$, $\lambda_2=\lambda_2'$;
\item $(ux,y)'=(x,\tau_1'(u)y)'$ for all $x,y\in \dotU$ and $u\in \UU$;
\item $(x^-1_\lambda, x'^- 1_\lambda)'=\bar{(\bar x,\bar{x'})}$ for all $x,x'\in \ff$ and all $\lambda$;
\item We have $(x,y)'=(y,x)'$.
\end{enumerate}

Moreover, we have $(\omega(x),\omega^{-1}(y))'=(\rho(x),\rho(y))'=(x,y)'$.
\end{proposition}

\begin{proof}
For $x,y\in \dotU$, set $(x,y)'=\bar{(\bar x^\dagger,\bar y^\dagger)}^\dagger$. This gives us a $\Qvp$-bilinear pairing which
clearly satisfies (1) and (4). (2) follows from the observation that $\bar{\tau_1}'(\bar{u}^\dagger)^\dagger=\tau_1(u)$.

For (3), by Theorem \ref{thm:dotbilinearform} (3) we have $(x^-1_\lambda, x'^-1_\lambda)'=(x^\dagger,x'^\dagger)^\dagger$.
However, it is easy to check that the bilinear form $(\cdot,\cdot)':\ff\times \ff\rightarrow \Qvp$ defined by 
$(x,y)'=(x^\dagger,y^\dagger)^\dagger$ for $x,y\in \ff$ satisfies the defining properties of $(\cdot,\cdot)$, hence
$(x,y)'=(x,y)$.

Lastly, we note that $\omega(\bar u^\dagger)=\bar{\omega^{-1}(u)}^\dagger$ and $\rho(\bar u^\dagger)=\bar{(\omega\rho\omega^{-1})(u)}^\dagger$,
so $\rho$- and $\omega$-invariance follows from the properties of $(\cdot,\cdot)$.
\end{proof}

\begin{remark}\label{remark:dagbilform}
Suppose $I=I_\one=\set{i}$.  In \cite{CW}, the bilinear form $(\cdot,\cdot)'$ is defined on $\dotU$.
With respect to this form, $(E^{(a)}1_n, E^{(a)}1_n)'$ is not $\pi$-almost-orthonormal (with respect
to $v^{-1}$) in general, a fact which is not desirable from a categorification standpoint. 
(However, in light of Remark \ref{remark:altdefquantgp}, 
Proposition \ref{prop:altdotbilform} demonstrates that $(\cdot,\cdot)'$ is well suited to the $\dagger$-twisted $\dotU$.)
In this regard, the bilinear form defined in Theorem \ref{thm:dotbilinearform} is a better choice.

Alternatively, Ellis and Lauda \cite{EL} have used a dagger-sesquilinear variant of $(\cdot,\cdot)'$ in their categorification.
Specifically, the Ellis-Lauda form $(\cdot,\cdot)_{\rm EL}$
seems to relate to ours via \[(x,y)_{\rm EL}=(\omega(x)^\dagger,\omega^{-1}(y))'.\]
\end{remark}

\subsection{The bilinear form as a limit}

Let $\zeta\in X$ and $\lambda,\lambda'\in X^+$ such that $\lambda-\lambda'=\zeta$.
From \cite{CHW2}, the module $V(\lambda)$ is equipped with a bilinear form
$(\cdot,\cdot)_\lambda$ such that \[(ux,y)_\lambda=(x,\tau_1(u)y)_\lambda\text{ for all } x,y\in V(\lambda),u\in \UU.\]
Such a bilinear form is called a polarization.
It is also shown in {\em loc. cit.} that while there is no natural polarization
on $V(\lambda)\otimes V(\lambda')$, there is an induced bilinear pairing
between $V(\lambda)\otimes V(\lambda')$ and $V(\lambda)\otimes_4 V(\lambda')$,
called a $J$-polarization; that is, a bilinear form 
$(\cdot,\cdot):(V(\lambda)\otimes V(\lambda'))\times (V(\lambda)\otimes V(\lambda'))\rightarrow \Qvp$ satisfying
\[(\Delta_3(u)x,y)=(x,\Delta_4(\tau_1(u))y)\text{ for all } x,y\in V(\lambda),u\in \UU.\]
We will use the following variant of this construction.
\begin{proposition}
Consider the bilinear pairing $(\cdot,\cdot)_{\lambda,\lambda'}: N(\lambda,\lambda')\times N(\lambda,\lambda')\rightarrow \Qvp$ 
defined by setting
\[(x\otimes x', y\otimes y')_{\lambda,\lambda'}=\pi_{\nu}\pi^{\ang{\tilde\nu,\lambda'}}(x,y)_\lambda(x',y')_{\lambda'}
\]
for homogeneous
$x,x'\in V(\lambda)$ and $y, y'\in V(\lambda')$ and $y\in V(\lambda')_{\lambda-\nu'}$
for some $\nu\in \N[I]$. Then $(\cdot,\cdot)_{\lambda,\lambda'}$
is a $J$-polarization.
\end{proposition}

\begin{proof}
This is an elementary verification akin to \cite[Proof of Lemma 4.9]{CHW2}; we note that
the $\pi_{\nu}\pi^{\ang{\tilde \nu,\lambda'}}$ factor is related to Lemma \ref{lem:dotbilform on Uplus}.
\end{proof}

\begin{proposition}
Let $x,y\in \dotU 1_\zeta$. When the pair $\lambda,\lambda'$ tends to $\infty$ (in the sense that $\ang{i,\lambda}$
tends to $\infty$ for all $i$), the inner product $(x(\eta_{\lambda}\otimes \zeta_{-\lambda'}),y(\eta_{\lambda}\otimes \zeta_{-\lambda'}))_{\lambda,\lambda'}$
converges in $\Q^\pi((v))$ to $(x,y)$.
\end{proposition}

\begin{proof}
Assume first that $x=x_1^-1_\zeta$ and $y=y_1^-1_\zeta$ for $x_1,y_1\in \ff$
Then \[x(\eta_{\lambda}\otimes \zeta_{-\lambda'})=x_1^-\eta_{\lambda}\otimes \zeta_{-\lambda'},\]
\[y(\eta_{\lambda}\otimes \zeta_{-\lambda'})=y_1^-\eta_{\lambda}\otimes \zeta_{-\lambda'}.\]
Therefore 
$(x(\eta_{\lambda}\otimes \zeta_{-\lambda'}),y(\eta_{\lambda}\otimes \zeta_{-\lambda'}))_{\lambda,\lambda'}=(x\eta_\lambda,y\eta_\lambda)_\lambda$,
and the right-hand side converges to $(x,y)=(x^-1_\zeta,y^-1_\zeta)$ by \cite[Proposition 6.10]{CHW2}.

Now assume that $x=1_\zeta$ and $y$ is arbitrary. Then we may assume that $y=\tau_1(x_1^-)y_1^- 1_\zeta$ for some $x_1,x_2\in \ff$.
But then by the polarization property, 
\[(1_\zeta (\eta_{\lambda}\otimes \zeta_{-\lambda'}), y(\eta_{\lambda}\otimes \zeta_{-\lambda'}))_{\lambda,\lambda'}
=(x_1^-(\eta_{\lambda}\otimes \zeta_{-\lambda'}), y_1^-(\eta_{\lambda}\otimes \zeta_{-\lambda'})).\]
Using the previous case, this converges to 
\[(x_1^-1_\zeta,y_1^-1_\zeta)=(1_\zeta,\tau_1(x_1^-)y_1^-1_\zeta)=(x,y).\]

Finally, let us assume $x$ and $y$ are both arbitrary. We may assume that $x=u1_\zeta$ for some $u\in \UU$.
Then 
\[(x (\eta_{\lambda}\otimes \zeta_{-\lambda'}), y(\eta_{\lambda}\otimes \zeta_{-\lambda'}))_{\lambda,\lambda'}
=(u(\eta_{\lambda}\otimes \zeta_{-\lambda'}), y(\eta_{\lambda}\otimes \zeta_{-\lambda'}))_{\lambda,\lambda'}
=(1_\zeta(\eta_{\lambda}\otimes \zeta_{-\lambda'}), \tau_1(u)y(\eta_{\lambda}\otimes \zeta_{-\lambda'})).\]
Again, by the previous case, this converges to $(1_\zeta,\tau_1(u)y)=(x,y)$.
\end{proof}

Given $\Qvp$-modules $M,M'$ and a pairing $(\cdot,\cdot):M\times M'\rightarrow \Qvp$,
we say a $\pi$-basis $B$ of $M$ is {\em $\pi$-almost-orthonormal} to a $\pi$-basis $B'$ of $M'$  if
they satisfy the following conditions:
\[(b,b')\in \Zp[[v]]\cap \Qvp;\]
\[(b,b')\in v\Zp[[v]]\text{ if }b\neq b'\text{ and }b\neq \pi b';\]
\[(b,b)\in \pi^{\epsilon}+ v\Zp[[v]]\text{ for some } \epsilon\in\set{0,1}.\]
In the case $M=M'$, we will just say $\pi$-almost-orthonormal.
\begin{theorem}
The basis $\dot\cB$ is $\pi$-almost-orthonormal.
\end{theorem}
\begin{proof}
It is trivial that $(b\diamondsuit_\zeta b',b_1\diamondsuit_{\zeta'} b_1')=0$ if $\zeta\neq \zeta'$.
In particular, we may assume the root datum is simply connected.

Recall by construction and by Corollary \ref{cor:dotCB Delta4 action} that for $s=3,4$ and
for all $\lambda,\lambda'\in X^+$ such that if $\lambda-\lambda'=\zeta$,
\[\Delta_{s}(b\diamondsuit_\zeta b')\eta_{\lambda}\otimes\xi_{-\lambda'}=\pi^{\epsilon}(b\diamondsuit_{s}b')_{\lambda,\lambda'},\]
where $\epsilon\in\set{0,1}$ (and by convention, we set $(b\diamondsuit_{s}b')_{\lambda,\lambda'}=0$
if $b\notin \cB(\lambda)$ or $b'\notin \cB(\lambda')$).

By the previous proposition, it suffices to show that 
for all $\lambda,\lambda'\in X^+$ such that if $\lambda-\lambda'=\zeta$ such that $b\in B(\lambda)$ and $b'\in B(\lambda')$,
$(b\diamondsuit b')_{\lambda,\lambda'}$ is $\pi$-almost-orthonormal to $(b_1\diamondsuit_4 b_1')_{\lambda,\lambda'}$
for all $b_1,b_1'$. But this follows from the definition of the bilinear form and the almost-$\pi$-orthonomality
of the canonical basis of $V(\lambda)$; cf. \cite[Proposition 6.1]{CHW2}.
\end{proof}

\begin{remark}
It should be noted that, in contrast to the non-super case, 
a characterization of $\dot \cB$ via $\pi$-almost-orthonormality is not possible,
since in particular it is not possible for $\cB$; cf. \cite[Example 6.3]{CHW2}.
\end{remark}

\section{Twistors}

In this section, we turn to some related constructions on $\dotU$ which
first appeared in \cite{CFLW}. Henceforth, we will assume that our
datum is also $X$-regular (as defined in \S \ref{subsec:datum}).
Let $\dot\Tw$ be the twistor map on $\dotU$ in \cite{CFLW}.
Then the complexification of $\dotU|_{\pi=-1}$ has a basis
obtained by applying $\dot\Tw$ to Lusztig's canonical basis
of $\dotU|_{\pi=1}$. We shall show that the resulting basis
equals the specialization of $\cB$ up to a factor of
an integral power of $\sqrt{-1}$.
Along the way, we extend the notion
of twistor maps to representations of $\UU$. 

\subsection{The twistor maps}
In this section, we recall (and restate) some results from \cite{CFLW}.
Let $\bt$ be a square root of $-1$. For $n,m\in \Z$, we shall use the shorthand $n\equiv_4 m$ 
if $n$ is congruent to $m$ modulo 4. 
We will freely use the fact that if $n,m\in \Z$ such that $n\equiv_4 m$,
then $\bt^n=\bt^m$.

\begin{definition} \label{definition:enhancer}
An {\em enhancer} $\phi$ is an function $\phi:\Z[I]\times X\rightarrow \Z$ satisfying
\begin{enumerate}
\item $\phi(\nu,\lambda+\mu')\equiv_4\phi(\nu,\mu')+\phi(\nu,\lambda)$ and $\phi(\nu+\mu,\lambda)\equiv_4\phi(\nu,\lambda)+\phi(\mu,\lambda)$ for $\nu,\mu\in \Z[I]$ and
$\lambda\in X$.
\item $\phi(i,j')\in 2\Z$ for $i\neq j\in I$.
\item $\phi(i,j')-\phi(j,i')\equiv_4 i\cdot j+2p(i)p(j)$ and 
$\phi(i,i')\equiv_4 d_i$ for $i\neq j\in I$.
\end{enumerate} 
\end{definition}
We note that $\phi$ need not be $\Z$-linear in the second coordinate; 
specifically, it need not be the case that $\phi(\nu,\lambda+\lambda')\equiv_4 \phi(\nu,\lambda)+\phi(\nu,\lambda')$ 
for $\lambda,\lambda'\in X$.

\begin{example}
Let $I=I_\one=\set{i}$, so $X=\Z$ and $i'=2\in \Z$. Then an enhancer $\phi$ is determined entirely by choosing a value of $\phi(1,1)$;
indeed, if $n,k\in \Z$ and $r\in\set{0,1}$ then \[\phi(n,2k+r)\equiv_4 n\phi(1,2k+r)=n\phi(1,r)+nk\phi(1,2),\]
and we have $\phi(1,2)\equiv_4 1$ and $\phi(1,0)\equiv_4 0$ from the definition of an enhancer. In this case, there is no way
for $\phi$ to be $\Z$-bilinear since $1\equiv_4 \phi(1,2)\not\equiv_4 2\phi(1,1)$ for any choice of $\phi(1,1)$.
\end{example}

As shown in {\em loc. cit.}, such an enhancer exists when the Cartan datum is $X$-regular.
We shall henceforth assume an enhancer exists, and fix a choice $\phi$.
We can extend scalars and enlarge the Cartan subalgebra of the quantum covering
group as follows.

\begin{definition} {\bf \cite[\S 4.4]{CFLW}}
 \label{definition:ecqg}
The {\em $\phi$-enhanced quantum covering group}
$\hatU$ associated to a super root datum $(Y, X, I, \cdot)$ and enhancer $\phi$ is the $\Q^\pi(\bt,v)$-algebra with generators
$E_i, F_i$, $K_\mu$, $J_\mu$, $T_\mu$, and $\Upsilon_\mu$, for  $i\in I$ and $\mu\in Y$, subject to the 
relations \eqref{eq:JKrels}-\eqref{eq:Fserrerel}, as well as the relations
\begin{equation}\label{eq:TUpsrels}
T_\mu T_\nu=T_{\mu+\nu},\quad \Upsilon_\mu \Upsilon_\nu=\Upsilon_{\mu+\nu},\quad T_0=\Upsilon_0=T_\nu^4=\Upsilon_\nu^4=1,\quad
T_\mu \Upsilon_\nu=\Upsilon_\nu T_\mu,
\end{equation} 
\begin{equation}\label{eq:Tweightrels}
T_\mu E_i=\bt^{\ang{\mu,i'}} E_i T_\mu,\quad T_\mu F_i=\bt^{-\ang{\mu,i'}} F_i T_\mu,
\end{equation} 
\begin{equation}\label{eq:Upsweightrels}
\Upsilon_\mu E_i=\bt^{\phi(\mu,i')} E_i \Upsilon_\mu,\quad \Upsilon_\mu F_i=\bt^{-\phi(\mu,i')} F_i \Upsilon_\mu,
\end{equation} 
\end{definition}

Note that $\dotU[\bt]$ is a $\hatU$-bimodule under the $\U$-module
action and under $T_\mu\Upsilon_\nu 1_\lambda=1_\lambda T_\mu\Upsilon_\nu= \bt^{\ang{\mu,\lambda}+\phi(\nu,\lambda)} 1_\lambda$.
The enhanced quantum covering group has a useful $\Q(\bt)$-linear automorphism called a {\em twistor}.
There are several ways to define such a twistor; we will need the following.

\begin{proposition}{\bf \cite[Theorems 4.3, 4.12]{CFLW}}
\begin{enumerate}
\item There is an automorphism $\Tw:\hatU\rightarrow \hatU$ defined by
\[\Tw(E_i)=\bt_i^2 E_i\tT_i \Upsilon_i,\quad \Tw(F_i)=F_i\Upsilon_{-i},\quad \Tw(K_\mu)=T_{-\mu}K_\mu,\quad \Tw(J_\mu)=T_{2\mu} J_\mu,\]
\[\Tw(T_\mu)=T_\mu,\quad \Tw(\Upsilon_\mu)=\Upsilon_\mu,\quad \Tw(v)=\bt^{-1} v,\quad \Tw(\pi)=-\pi,\]
where if $\mu=\sum_{i\in I} \mu_i i$, $\tT_\mu=\prod_{i\in I} T_{\mu_i d_i i}$.
\item There is an automorphism ${\Tw}:\dotU[\bt]\rightarrow \dotU[\bt]$ such that $\Tw(u1_\lambda u')=\Tw(u)1_\lambda\Tw(u')$ 
for $u,u'\in \hatU$.
\end{enumerate}
\end{proposition}

\begin{proof}
Though we are using a modified version of the twistor defined in \cite{CFLW},
essentially the same arguments appearing in {\em loc. cit.} work in this case. 
For completeness, we will provide some arguments for (2) here.

All the relations are clear except for \eqref{eq:commutatorrelation}, \eqref{eq:Eserrerel} and \eqref{eq:Fserrerel}.
The Serre relations can be proved as in \cite{CFLW}, or alternatively one may use \eqref{eq:twist vs ff-embeddings}
below. For the commutator relation,
\begin{align*}
&\bt_i^2E_i\tT_i\Upsilon_i F_j\Upsilon_{-j}-(-\pi)^{p(i)p(j)} F_j\Upsilon_{-j}\bt_i^2E_i\tT_i\Upsilon_i\\
&\hspace{2em}=\bt^{2d_i-\phi(i,j)+i\cdot j}(E_iF_j-\pi^{p(i)p(j)}\bt^{2p(i)p(j)-i\cdot j+\phi(i,j)-\phi(j,i)}F_jE_i)\tT_i\Upsilon_{i-j}\\
&\hspace{2em}=\bt^{2d_i-\phi(i,j)+i\cdot j}(E_iF_j-\pi^{p(i)p(j)}F_jE_i)\tT_i\Upsilon_{i-j}\\
&\hspace{2em}=\delta_{ij} \bt^{-1}\frac{\tJ_i\tK_i-\tK_i^{-1}}{\pi_iv_i-v_i^{-1}}\tT_i=\delta_{ij}\bt^{-1}\frac{\tT_i^2\tJ_i\tT_i^{-1}\tK_i-\tT_i\tK_i^{-1}}{(-\pi_i\bt_i^{-1}v_i-\bt_iv_i^{-1}}.
\end{align*}
\end{proof}

There is another version of the twistor defined on $\ff[\bt]=\Q^\pi(v,\bt)\otimes_{\Qvp}\ff$.

\begin{proposition}{\bf \cite[Theorem 2.4]{CFLW}}\label{prop:halfquantumtwist}
Define a product $*$ on $\ff[\bt]$ by the following rule: if $x$ and $y$ are homogeneous elements of $\ff[\bt]$,
let $x*y=\bt^{\phi(|x|,|y|')}xy$. Then there is a $\Q(\bt)$-linear algebra isomorphism $\Tw:\ff\rightarrow (\ff,*)$
defined by
\[\Tw(\theta_i)=\theta_i,\quad \Tw(v)=\bt^{-1}v,\quad \Tw(\pi)=-\pi.\]
Moreover, $\Tw$ satisfies $\Tw(b)=\bt^{\ell(b)} b$ for all $b\in \cB(\infty)$,
where $\ell(b)$ is some integer depending on $b$.
\end{proposition}

Now note that the maps $\theta_i\mapsto \theta_i^-=F_i$ and $\theta_i\mapsto \theta_i^+=E_i$ extend
to embeddings $\ff[\bt]\rightarrow \hatU^-$ and $\ff[\bt]\rightarrow \hatU^+$. These embeddings relate
the twistor maps as follows: for $\nu=\sum_{i\in I} \nu_i i\in \N[I]$ 
and $x\in \ff[\bt]_{\nu}$,
\begin{equation}\label{eq:twist vs ff-embeddings}
\Tw(x^-)=\Tw(x)^-\Upsilon_{-\nu},\qquad 
\Tw(x^+)=\bt^{\bullet(\nu)+\sum 2d_i\nu_i}\Tw(x)^+\tT_{|x|}\Upsilon_{|x|}. 
\end{equation}

These equations imply the following result.
\begin{lemma}\label{lem:tw and omega CB}
Let $b\in \cB$. Then $\Tw(\omega(b^-))=\bt^{\ell(b)+\bullet(\nu)}\omega(b^-)
\Upsilon_\nu \tT_{-\nu}$.
\end{lemma}

\begin{proof}
Since $\Tw(x^+)=\bt^{\bullet(\nu)+\sum 2d_i\nu_i}\Tw(x)^+\tT_{\nu}\Upsilon_{\nu}$
for $x\in \ff[\bt]_\nu$, then in particular for $b\in \cB(\nu)_{\nu}$,
\[
\Tw(\omega(b^+))=\Tw(b^+\pi_\nu \tJ_{\nu})=\bt^{\ell(b)+\bullet(\nu)+\sum 2d_i\nu_i}(-1)^{\sum \nu_i d_i}\pi_\nu\tJ_\nu b^+
\Upsilon_\nu \tT_{-\nu}.\]

\end{proof}

Recall the $\Qvp$-linear differentials ${}_ir,r_i:\ff\rightarrow \ff$ defined by ${}_i r(\theta_j)=r_i(\theta_j)=\delta_{ij}$,
\[{}_i r(xy)={}_ir(x)y+\pi^{p(x)p(i)}v^{-i\cdot |x|}x\ {}_ir(y),\qquad r_i(xy)=\pi^{p(x)p(i)}v^{-i\cdot |x|}r_i(x)y+xr_i(y).\]
These maps trivially extend to $\Q^\pi(v,\bt)$-linear maps on $\ff[\bt]$.

\begin{lemma}
For any $x\in \ff[\bt]_\nu$,
\[\Tw({}_i r(x))=\bt^{-\phi(i,\nu'-i')}{}_i r(\Tw(x)),\quad \Tw(r_i(x))=\bt^{-\phi(\nu-i,i')}r_i(\Tw(x)).\]
\end{lemma}
\begin{proof}
This is a straightforward verification.
\end{proof}

\subsection{Twistors on simple modules}

Now we will expand on some additional results we shall need regarding the twistors and highest weight modules.
Recall from \cite{CHW1} that the Verma module $M(\lambda)$
is the vector space $\ff$ with module structure defined as follows: for homogeneous $x\in \ff$,
\[K_\mu 1=v^{\ang{\mu,\lambda-|x|'}} x,\quad J_\mu x=\pi^{\ang{\mu,\lambda-|x|'}} x,\quad F_i x=\theta_ix,\]
\[ E_i x=
\frac{\pi_i^{p(x)-p(i)}r_i(x)(\pi_iv_i)^{\ang{i,\lambda}}   - v_i^{-\ang{i,\lambda-|x|'+i'}}\ {}_ir(x)}{\pi_i v_i-v_i^{-1}}.\]

\begin{lemma}
Let $\Tw_\lambda:M(\lambda)[\bt]\rightarrow M(\lambda)[\bt]$ be the map defined by
\[\Tw_\lambda(x)=\bt^{-\phi(\nu,\lambda)} \Tw(x)\]
for homogeneous $x\in \ff[\bt]_{\nu}$.
Then $\Tw_\lambda(ux)=\Tw(u)\Tw_\lambda(x)$ for all $u\in \hatU$ and $x\in \ff[\bt]$.
\end{lemma}
\begin{proof}
It is enough to prove the lemma for the generators of $\hatU$.
First note that $\Tw_\lambda$ preserves weight spaces, and so the lemma is clear for $T_\mu$ and $\Upsilon_\mu$;
it is also clear that $\Tw(vx)=\Tw(v)\Tw_\lambda(x)$ and $\Tw(\pi x)=\Tw(\pi)\Tw_\lambda(x)$, whence $\Tw_\lambda(K_\mu x)=\Tw(K_\mu)\Tw_\lambda(x)$
and $\Tw_\lambda(J_\mu x)=\Tw(J_\mu)\Tw_\lambda(x)$. It remains to check for $u=E_i$ and $u=F_i$.

Let $x\in \ff[\bt]_{\nu}$. Then for $u=F_i$ we have
\[\Tw_\lambda(F_ix)=\Tw_\lambda(\theta_i x)=\bt^{\phi(i,\nu')-\phi(\nu+i,\lambda)} \theta_i\Tw(x)= F_i\Upsilon_{-i}\Tw_\lambda(x)=\Tw(F_i)\Tw_\lambda(x).\]
On the other hand, for $u=E_i$ we have
\begin{align*}
&\Tw_\lambda(E_ix)
=\Tw_\lambda\parens{\frac{\pi_i^{p(\nu)-p(i)}r_i(x)(\pi_iv_i)^{\ang{i,\lambda}}   - v_i^{-\ang{i,\lambda-\nu'+i'}}\ {}_ir(x)}{\pi_i v_i-v_i^{-1}}}\\
&=\bt^{-d_i-\phi(\nu-i,\lambda)}
\frac{\bt_i^{2p(i)p(\nu-i)+\ang{i,\lambda}}\pi_i^{p(\nu)-p(i)}\Tw(r_i(x))(\pi_iv_i)^{\ang{i,\lambda}}   
- \bt_i^{\ang{i,\lambda-\nu+i'}}v_i^{-\ang{i,\lambda-\nu+i'}}\Tw({}_ir(x))}
{\pi_i v_i-v_i^{-1}}\\
&=\bt^{-d_i-\phi(\nu-i,\lambda)+d_i\ang{i,\lambda-\nu+i'}-\phi(i,\nu'-i')}
\frac{\bt_i^{\bigstar}\pi_i^{p(\nu)-p(i)}r_i(\Tw(x))(\pi_iv_i)^{\ang{i,\lambda}}   
- v_i^{-\ang{i,\lambda-\nu+i'}}{}_ir(\Tw(x))}
{\pi_i v_i-v_i^{-1}}
\end{align*}
where $\bigstar=2p(i)p(\nu-i)+d_i\ang{i,\lambda}-\phi(\nu-i,i)-d_i\ang{i,\lambda-\nu+i'}+\phi(i,\nu'-i')\equiv_4 0$.
Therefore,
\begin{align*}
\Tw_\lambda(E_ix)&
=\bt^{-d_i-\phi(\nu-i,\lambda)+d_i\ang{i,\lambda-\nu+i'}-\phi(i,\nu'-i')}E_i\Tw(x)\\
&=\bt_i^{\ang{i,\lambda-\nu}}\bt^{\phi(i,\lambda-\nu)}\bt^{-\phi(\nu,\lambda)+2d_i}E_i\Tw(x)\\
&=\bt_i^2E_i\tT_i\Upsilon_i \Tw_\lambda(x)=\Tw(E_i)\Tw_\lambda(x).
\end{align*}
\end{proof}

Recall that $V(\lambda)=\ff/(\theta_i^{\ang{i,\lambda}+1}:i\in I)$
as vector spaces. Then $V(\lambda)[\bt]=\ff[\bt]/(\theta_i^{\ang{i,\lambda}+1}:i\in I)$.
Since we further have
\[\Tw_\lambda((\theta_i^{\ang{i,\lambda}+1}:i\in I))=(\theta_i^{\ang{i,\lambda}+1}:i\in I)\]
we see that $\Tw_\lambda$ induces a $\Q(\bt)$-linear isomorphism
\[\Tw_{\lambda}: V(\lambda)[\bt]\rightarrow V(\lambda)[\bt].\]

\begin{lemma}\label{lem:modtwistor}
There is a $\Q(\bt)$-linear map $\Tw_\lambda:V(\lambda)[\bt]\rightarrow V(\lambda)[\bt]$ 
which satisfies $\Tw_\lambda(\eta_\lambda)=\eta_\lambda$ and $\Tw_\lambda(um)=\Tw(u)\Tw_\lambda(m)$ for all $u\in \hatU$
and $m\in V(\lambda)[\bt]$.
\end{lemma}

\subsection{Twistors on tensors}

Let $a,a',P(\lambda-\lambda',a,a')$ be as in \S\ref{subsec:dotU and N}. By
the results of that section, we have an identification 
\[\dotU 1_{\lambda-\lambda'}/P(\lambda-\lambda',a,a')= N(\lambda,\lambda')\]
as $\UU$-modules. It is easy to see that $\dot\Tw$ preserves $\dotU1_{\lambda-\lambda'}$ and $P(\lambda-\lambda',a,a')$, so in particular 
$\dot\Tw:\dotU\rightarrow \dotU$ induces a $\UU$-module homomorphism
$\Tw_{\lambda,\lambda'}:N(\lambda,\lambda')\rightarrow N(\lambda,\lambda')$ such that 
\[\Tw_{\lambda,\lambda'}(\Delta(u) \eta_\lambda\otimes \xi_{-\lambda})=\Delta(\Tw(u)) \eta_\lambda\otimes \xi_{-\lambda}.\]

To understand this map better, first we must construct an analogue of Lemma \ref{lem:modtwistor} for
${}^\omega V(\lambda)[\bt]$. To that end,
set $\Tw'=\omega^{-1}\circ \Tw \circ \omega$; we note that $\Tw'(E_i)=\tT_{-i}\Tw(E_i)$, $\Tw'(F_i)=\Tw(F_i)\tT_i$,
$\Tw'(K_\nu)=\Tw(K_\nu)$, $\Tw'(J_\nu)=\Tw(J_\nu)$, $\Tw'(T_\nu)=\Tw(T_\nu)$.
We note that in particular,
\[\Delta(\Tw(E_i))=\Tw(E_i)\otimes \Tw'(\tK_i)\Upsilon_i+\bt_i^2(\Upsilon_i\otimes \Tw'(E_i))(\tT_i\otimes \tT_i),\]
\[\Delta(\Tw(F_i))=\Tw(F_i)\otimes \Upsilon_{-i}+(\Upsilon_{-i}\Tw(\tJ_i\tK_i)\otimes \Tw'(F_i))(\tT_{-i}\otimes \tT_{-i}).\]

There is a $\Q(\bt)$-linear map 
$\Tw_{-\lambda}:{}^\omega V(\lambda)[\bt]\rightarrow {}^\omega V(\lambda)[\bt]$ 
which satisfies $\Tw_\lambda(\xi_{-\lambda})=\xi_{-\lambda}$ and $\Tw_{-\lambda}(um)=\Tw'(u)\Tw_{-\lambda}(m)$ for all $u\in \hatU$
and $m\in {}^\omega V(\lambda)[\bt]$; indeed, view $m\in {}^\omega V(\lambda)[\bt]$ as an element of $V(\lambda)[\bt]$
as an element of $V(\lambda)$ and set $\Tw_{-\lambda}(m)=\Tw_{\lambda}(m)$. Then
\[\Tw_{-\lambda}(u\cdot m)=\Tw_\lambda(\omega(u)m)=\Tw(\omega(u)) \Tw_{\lambda}(m)=\Tw'(u)\cdot \Tw_{-\lambda}(m).\]

Then given $w\otimes w'\in N(\lambda,\lambda')$ with $|w|=\zeta$ and $|w'|=\zeta'$, we see that
\begin{align}
\label{eq:deltwE}\Delta(\Tw(E_i))(\Tw_\lambda(w)\otimes \Tw_{-\lambda'}(w')&=\bt^{\phi(i,\zeta')}\Tw_\lambda(E_iw)\otimes \Tw_{-\lambda'}(\tK_{-i}\cdot w')\\
\notag&\hspace{2em}+\pi^{p(i)p(w)}\bt^{2d_i+d_i\ang{i,\zeta+\zeta'}+\phi(i,\zeta)}
\Tw_\lambda(w)\otimes \Tw_{-\lambda'}(E_i\cdot w')
\end{align}
\begin{align}
\label{eq:deltwF}
\Delta&(\Tw(F_i))(\Tw_\lambda(w)\otimes \Tw_{-\lambda'}(w')\\
\notag&=\bt^{-\phi(i,\zeta')}\Tw_\lambda(F_iw)\otimes \Tw_{-\lambda'}( w')+\pi^{p(i)p(w)}\bt^{-d_i\ang{i,\zeta-\zeta'}-\phi(i,\zeta)}
\Tw_\lambda(\tJ_i\tK_i w)\otimes \Tw_{-\lambda'}(F_i\cdot w')
\end{align}

Let ${\rm \Lambda}_\lambda=\set{\lambda-\nu'\mid \nu\in \N[I]}$ and ${\rm V}_\lambda=\set{\nu'-\lambda\mid \nu\in \N[I]}$.
and for $\zeta\in {\rm \Lambda}_\lambda$ (resp. $\zeta\in {\rm V}_\lambda$) such that $\zeta=\lambda-\nu$ (resp. $\zeta=\nu-\lambda$), 
set $p(\zeta)=p(\nu)$.

\newcommand{\cX}{\varkappa}

\begin{lemma}
There is a function $\cX=\cX_{\lambda,\lambda'}:{\rm \Lambda}_\lambda\times {\rm V}_{\lambda'}\rightarrow \Z$
satisfying \begin{enumerate}
\item $\cX(\zeta-i',\zeta')\equiv_4 \cX(\zeta,\zeta')-\phi(i,\zeta')$;
\item $\cX(\zeta,\zeta'+i)\equiv_4 \cX(\zeta,\zeta')+\phi(i,\zeta)+2d_i+\ang{\tilde{i},\zeta+\zeta'}+2p(\zeta)p(i)$;
\item $\cX(\lambda,-\lambda')\equiv_4 0$.
\end{enumerate}
\end{lemma}

\begin{proof}
Let $\mu,\nu\in \N[I]$ and write $\zeta=\lambda-\nu$, $\zeta'=\mu-\lambda'$.
Set 
\[\cX(\lambda,-\lambda')=0, \qquad \cX(\lambda-\nu',-\lambda')=\phi(\nu,\lambda').\]
Then we have defined $\cX(\lambda-\nu',\mu'-\lambda')$ for $\height(\mu)= 0$;
now assume $\height(\mu)>0$.
Define
\[\cX(\lambda-\nu',i'+\mu'-\lambda')=\cX(\lambda-\nu',\mu'-\lambda')+\phi(i,\lambda-\nu')+2d_i+\ang{\tilde{i},\lambda-\nu'+\mu'-\lambda'}+2p(\nu)p(i).\]
It is straightforward to check that this definition does not depend on the choice of $i$; that is, if $i+\mu=j+\hat \mu$ 
for some $\hat \mu\in \N[I]$,
then $\cX(\lambda-\nu',i'+\mu'-\lambda')=\cX(\lambda-\nu',j'+\hat \mu'-\lambda')$.

Now we need to check that $\cX(\lambda-\nu',\mu'-\lambda)$ satisfies (1) and (2).
By construction, we see that (1) and (2) hold for $\height(\nu+\mu)=0$, so
assume that $\height(\nu+\mu)>0$. By construction, (2) holds and so it suffices to check (1).
Write $\mu=\mu_1+j$ for some $j\in I$ and $\mu_1\in \N[I]$.
Then by induction we compute that
\begin{align*}
\cX(\lambda-\nu'-i',\mu'-\lambda)&\equiv_4\cX(\lambda-\nu'-i',\mu_1'-\lambda)+\phi(j,\lambda-\nu'-i')+2d_j\\
&\hspace{2em}+\ang{\tilde{j},\lambda-\nu'-i'+\mu_1'-\lambda'}+2p(\nu+i)p(i)\\
&\equiv_4\cX(\lambda-\nu',\mu_1'-\lambda)-\phi(i,\mu_1'-\lambda')+\phi(j,\lambda-\nu'-i')+2d_j\\
&\hspace{2em}+\ang{\tilde{j},\lambda-\nu'-i'+\mu_1'-\lambda'}+2p(\nu+i)p(i)\\
&\equiv_4\cX(\lambda-\nu',\mu_1'-\lambda)+2p(\nu)p(j)+\phi(j,\lambda-\nu')+2d_j\\
&\hspace{2em}+\ang{\tilde{j},\lambda-\nu'+\mu_1'-\lambda'}-i\cdot j+2p(i)p(j)-\phi(j,i)-\phi(i,\mu_1'-\lambda')\\
&\equiv_4\cX(\lambda-\nu',\mu'-\lambda)-i\cdot j+2p(i)p(j)-\phi(j,i)-\phi(i,\mu_1'-\lambda')
\end{align*}

Now $-i\cdot j+2p(i)p(j)-\phi(j,i)\equiv_4-\phi(i,j)$, and thus we see that
\[\cX(\lambda-\nu'-i',\mu'-\lambda)\equiv_4\cX(\lambda-\nu',\mu'-\lambda)-\phi(i,\mu'-\lambda').\]
This finishes the proof.
\end{proof}
\begin{proposition}\label{prop:Ntwistor}
The $\Q(\bt)$-linear map $\Tw_{\lambda,\lambda'}:N(\lambda,\lambda')[\bt]\rightarrow N(\lambda,\lambda')[\bt]$
defined by 
\[\Tw_{\lambda,\lambda'}(w\otimes w')= \bt^{\cX(|v|,|w|)} \Tw_{\lambda}(w)\otimes \Tw_{-\lambda'}(w'),\]
where $|w|=\lambda-\nu'$ and $|w'|=\lambda'-\mu'$ for $\nu,\mu\in \N[I]$.
Then
$\Tw_{\lambda,\lambda'}(\Delta(u)w\otimes w')=\Delta(\Tw(u)) \Tw_{\lambda,\lambda'}(w\otimes w')$.
\end{proposition}

\begin{proof}
Let $w\otimes w'\in  N(\lambda,\lambda')$ with $|w|=\zeta$ and $|w'|=\zeta'$ for $\nu,\mu\in \N[I]$.
It is enough to check when $u$ is a generator. If $u$ is $K_\nu$, $J_\nu$, $T_\nu$, or $\Upsilon_\nu$ then
this is trivial, so it remains to check when $u=E_i$ or $u=F_i$.

Well, using \eqref{eq:deltwE}, we have
\begin{align*}
\Delta(\Tw(E_i))& \Tw_{\lambda,\lambda'}(w\otimes w')=
\bt^{\cX(\zeta,\zeta')}\Delta(\Tw(E_i))(\Tw_\lambda(w)\otimes \Tw_{-\lambda'}(w'))\\
&=\bt^{\phi(i,\zeta')+\cX(\zeta,\zeta')}\Tw_\lambda(E_iw)\otimes \Tw_{-\lambda'}(\tK_{-i}\cdot w')\\
&\hspace{2em} +\pi^{p(i)p(w)}\bt^{\cX(\zeta,\zeta')+2d_i+d_i\ang{i,\zeta+\zeta'}+\phi(i,\zeta)}
\Tw_\lambda(w)\otimes \Tw_{-\lambda'}(E_i\cdot w')\\
&=\bt^{\phi(i,\zeta')+\cX(\zeta,\zeta')-\cX(\zeta+i,\zeta')}\Tw_{\lambda,\lambda'}(E_iw\otimes \tK_{-i}\cdot w')\\
&\hspace{2em} +\pi^{p(i)p(w)}\bt^{\cX(\zeta,\zeta')-\cX(\zeta,\zeta'+i)+2d_i+d_i\ang{i,\zeta+\zeta'}+\phi(i,\zeta)}
\Tw_{\lambda,\lambda'}(w\otimes E_i\cdot w')\\
&=\Tw_{\lambda,\lambda'}(E_iw\otimes\tK_{-i}\cdot w') +\pi^{p(i)p(w)}\bt^{2p(i)p(\zeta)} \Tw_{\lambda,\lambda'}(w\otimes E_i\cdot w')\\
&=\Tw_{\lambda,\lambda'}(\Delta(E_i)w\otimes w').
\end{align*}
Similarly, using \eqref{eq:deltwF}, we have
\begin{align*}
\Delta(\Tw(F_i))& \Tw_{\lambda,\lambda'}(w\otimes w')=
\bt^{\cX(\zeta,\zeta')}\Delta(\Tw(F_i))(\Tw_\lambda(w)\otimes \Tw_{-\lambda'}(w'))\\
&=\bt^{-\phi(i,\zeta')+\cX(\zeta,\zeta')}\Tw_\lambda(F_iw)\otimes \Tw_{-\lambda'}( w')\\
&\hspace{2em} +\pi^{p(i)p(w)}\bt^{\cX(\zeta,\zeta')-d_i\ang{i,\zeta+\zeta'}-\phi(i,\zeta)}
\Tw_\lambda(\tJ_i\tK_i w)\otimes \Tw_{-\lambda'}(F_i\cdot w')\\
&=\bt^{-\phi(i,\zeta')+\cX(\zeta,\zeta')-\cX(\zeta-i',\zeta')}\Tw_{\lambda,\lambda'}(F_iw\otimes w')\\
&\hspace{2em} +\pi^{p(i)p(w)}\bt^{\cX(\zeta,\zeta')-\cX(\zeta,\zeta'-i')-d_i\ang{i,\zeta+\zeta'}-\phi(i,\zeta)}
\Tw_{\lambda,\lambda'}(\tJ_i\tK_i w \otimes F_i\cdot w')\\
&=\Tw_{\lambda,\lambda'}(\Delta(F_i)w\otimes w')
\end{align*}
\end{proof}

Then in particular we have the following result.

\begin{theorem}\label{thm:twCB} Let $b,b'\in \cB$ and $\lambda,\lambda'\in X^+$. Set $\zeta=\lambda-\lambda'$.
\begin{enumerate}
\item We have 
$\Tw_{\lambda,\lambda'}((b\diamondsuit b')_{\lambda,\lambda'})=\bt^{f(b,b',\zeta)}(b\diamondsuit b')_{\lambda,\lambda'}$
for some $f(b,b',\zeta)\in \Z$.
\item  We have
$\Tw(b\diamondsuit_\zeta b')=\bt^{f(b,b',\zeta)}(b\diamondsuit_\zeta b')$.
\end{enumerate}
\end{theorem}
\begin{proof}
For (1), first note that the claim $\Tw_{\lambda,\lambda'}((b\diamondsuit b')_{\lambda,\lambda'})=\bt^{f(b,b',\lambda,\lambda')}(b\diamondsuit b')_{\lambda,\lambda'}$ for some $f(b,b',\lambda,\lambda')\in \Z$ follows by combining Proposition \ref{prop:Ntwistor}, Proposition \ref{prop:halfquantumtwist} and \eqref{eq:twist vs ff-embeddings}.
Let $u=b\diamondsuit_\zeta b'$.
Then $\Tw(u(\eta_\lambda\otimes \xi_{-\lambda'}) )=\bt^{f(b,b',\lambda,\lambda')}(b\diamondsuit b')_{\lambda,\lambda'}$;
on the other hand, $\Tw(u(\eta_\lambda\otimes \xi_{-\lambda'}) )=\Delta(\Tw(u))(\eta_\lambda\otimes \xi_{-\lambda'})$.
Therefore, we see that 
\[\Delta(\Tw(u))(\eta_\lambda\otimes \xi_{-\lambda'})=\bt^{f(b,b',\lambda,\lambda')}(b\diamondsuit b')_{\lambda,\lambda'}\]
Let $\lambda''\in X^+$. Then
\begin{align*}
\Delta(\Tw(u))(\eta_\lambda\otimes \xi_{-\lambda'})&=\Delta(\Tw(u))t(\eta_{\lambda+\lambda''}\otimes \xi_{-\lambda''-\lambda'})\\
&=t(\Delta(\Tw(u))\eta_{\lambda+\lambda''}\otimes \xi_{-\lambda''-\lambda'})\\
&=t(\bt^{f(b,b',\lambda+\lambda'',\lambda''+\lambda')}(b\diamondsuit b')_{\lambda+\lambda'',\lambda''+\lambda'})\\
&=\bt^{f(b,b',\lambda+\lambda'',\lambda''+\lambda')}(b\diamondsuit b')_{\lambda,\lambda'}
\end{align*}
In particular, we see that $f(b,b',\lambda,\lambda')\equiv_4 f(b,b',\lambda+\lambda'',\lambda''+\lambda')$,
so $\bt^{f(b,b',\lambda,\lambda')}$ is determined by $b$, $b'$, and $\zeta=\lambda-\lambda'$, which finishes the proof of (1).

In particular, setting $f(b,b',\zeta)=f(b,b',\lambda,\lambda')$,
\[\Delta(\Tw(u))(\eta_{\lambda_0}\otimes \xi_{-\lambda_0'})=\bt^{f(b,b',\zeta)}(b\diamondsuit b')_{\lambda_0,\lambda'_0}\]
for all $\lambda_0,\lambda_0'\in X^+$ with $\zeta=\lambda_0-\lambda_0'$, so (2) follows by uniqueness.
\end{proof}

\end{document}